%% file: main.tex
\documentclass[11pt]{article}
\usepackage[margin=1.05in]{geometry}

\usepackage[english]{babel}
\usepackage{mathtools,amssymb,amsthm,mathrsfs}
\usepackage{xurl}
\usepackage{xcolor}
\numberwithin{equation}{section}
\usepackage[hidelinks,
  pdftitle={A Metric with Positive Sectional Curvature on S3 times S3},
  pdfauthor={Shengtao Guo, Ethan X. Fang, Junwei Lu}
]{hyperref}

\title{A Metric with Positive Sectional Curvature on \(S^3\times S^3\)}
\date{}
\author{
  Shengtao Guo \qquad
  Ethan X. Fang \qquad
  Junwei Lu\thanks{Department of Biostatistics, Harvard T.H. Chan School of
	Public Health. Email: \texttt{junweilu@hsph.harvard.edu}.}
}

\newtheorem{theorem}{Theorem}[section]
\newtheorem{proposition}[theorem]{Proposition}
\newtheorem{lemma}[theorem]{Lemma}

\theoremstyle{definition}

\theoremstyle{remark}
\newtheorem{remark}[theorem]{Remark}

\newcommand{\R}{\mathbb R}
\newcommand{\Q}{\mathbb Q}
\newcommand{\Z}{\mathbb Z}
\newcommand{\cZ}{\mathcal Z}
\newcommand{\cN}{\mathcal N}
\newcommand{\cL}{\mathcal L}
\newcommand{\cQ}{\mathcal Q}
\newcommand{\cR}{\mathcal R}
\newcommand{\Gr}{\operatorname{Gr}}
\newcommand{\Ad}{\operatorname{Ad}}
\newcommand{\Span}{\operatorname{span}}
\newcommand{\dist}{\operatorname{dist}}

\newcommand{\dd}{\mathrm d}
\newcommand{\eps}{\varepsilon}
\newcommand{\Sec}{\operatorname{sec}}

\newcommand{\mean}{\operatorname{mean}}
\newcommand{\Vtwo}{V^{(2)}}
\newcommand{\Vthree}{V^{(3)}}
\newcommand{\Vthreecorr}{\widetilde V^{(3)}}
\newcommand{\Vthreebase}{V^{(3)}_0}
\newcommand{\Tlift}{\widehat T_0}
\colorlet{coefficientcolor}{red}

\allowdisplaybreaks[2]

\begin{document}
\maketitle

\begin{abstract}
We construct a smooth Riemannian metric with strictly positive sectional
curvature on \(S^3\times S^3\). Since \(S^3\times S^3\) is
even-dimensional and has Euler characteristic zero, our result disproves
the positive-curvature case of Hopf's sign conjecture. The metric and the proof were discovered by the Odin Automatic AI Research Agent.
\end{abstract}

\input{introduction}
\input{reduction}
\input{tensors}
\input{background}
\input{quadratic}
\input{cubic}
\input{positivity}
\appendix
\input{calculations}

\bibliographystyle{alpha}
\bibliography{positive-curvature-s3xs3}
\end{document}

%% file: introduction.tex
\section{Introduction}

The relation between positive sectional curvature and topology is one of the
classical problems of global Riemannian geometry. Spheres and the compact
rank-one symmetric spaces provide the basic examples. The constructions of
positively curved homogeneous spaces and biquotients show that the class is
larger, but also illustrate how restrictive positivity on every tangent
two-plane can be; see Wallach~\cite{Wallach1972},
Eschenburg~\cite{Eschenburg1982}, and the survey~\cite{Ziller2007}.
 O'Neill's
formula~\cite{ONeill1966} and
Cheeger deformation~\cite{Cheeger1973} are fundamental tools for constructing positively curved metrics.

There are two famous conjectures attributed to Hopf about the relation between positive sectional curvature and topology. Hopf's product conjecture asserts that \(S^2\times S^2\) does not admit positive sectional curvature, which has been disproved by Brendle and Hung~\cite{BrendleHung2026}.
Hopf's sign conjecture has two parts: the positive-curvature case asserts that a closed
$2d$-dimensional manifold with positive sectional curvature has positive
Euler characteristic; and the negative-curvature case asserts that a closed $2d$-dimensional manifold with negative sectional curvature has Euler characteristic of sign $(-1)^d$.
The Gauss--Bonnet theorem settles the Euler characteristic assertion in dimension two,
and it also holds in dimension four. In dimension six, however, positivity
of sectional curvature does not imply pointwise positivity of the
integrand in the Gauss--Bonnet theorem: Geroch constructed an algebraic curvature tensor
with positive sectional curvatures and negative integrand~\cite{Geroch1976}.
This algebraic observation does not itself produce a closed manifold, but
explains why the four-dimensional argument cannot simply be repeated.

Symmetry has supplied strong positive results for the Hopf sign
conjecture. Wallach's work gives the assertion for positively curved
homogeneous spaces~\cite{Wallach1972}. Grove and
Searle~\cite{GroveSearle1994} classify positively curved manifolds of maximal
symmetry rank. P\"uttmann and Searle~\cite{PuttmannSearle2002} prove the Hopf
assertion under much weaker symmetry hypotheses; in particular, a closed
positively curved six-manifold admitting a nontrivial effective isometric
circle action has positive Euler characteristic. Kennard~\cite{Kennard2013}
establishes further Euler characteristic results under logarithmic symmetry
rank assumptions in higher dimensions. These theorems concern symmetries
of the positive metric, not merely group actions on the underlying manifold.

In dimension six, the classical simply connected positively curved examples
are \(S^6\), \(\mathbb{CP}^3\), the Wallach flag manifold
\(SU(3)/T^2\), and the Eschenburg biquotient \(SU(3)//T^2\)
\cite{Wallach1972,Eschenburg1992}. Their Euler characteristics are,
respectively, \(2,4,6,6\).
For closed, simply connected six-manifolds with positive sectional curvature
and an effective isometric \(SU(2)\)- or \(SO(3)\)-action,
Liu~\cite[Theorems~1.1--1.3]{Liu2024} proves that the Euler characteristic is one of the values
arising from these examples and obtains equivariant classifications under
additional hypotheses on the orbits. The circle-action result of
P\"uttmann and Searle already implies that a positively curved metric on
\(S^3\times S^3\) must have a finite isometry group, since its Euler
characteristic is zero. Any construction
on this product must ultimately break the continuous symmetries that make
its background metric tractable.

The manifold \(S^3\times S^3\) is therefore a natural test case for
Hopf's sign conjecture and the question of positive curvature on products.
Since \(\chi(S^3\times S^3)=\chi(S^3)\chi(S^3)=0\), the metric constructed
here disproves the positive-curvature case of Hopf's sign conjecture.
It also gives another example of a positively curved metric on a product
of two spheres, following the examples on \(S^2\times S^2\)
\cite{BrendleHung2026} and \(S^2\times S^3\) \cite{GuoFangLu2026}.
The most straightforward metric on \(S^3\times S^3\) is the round product
metric, which is nonnegatively curved, but
all mixed two-planes are flat. A more useful background comes from the
normal homogeneous presentation \(SU(2)^3/\Delta SU(2)\).
Spatzier and Strake~\cite[Remark~2.4]{SpatzierStrake1990} observed that this
metric has rank one in the sense of parallel Jacobi fields, although it
still contains totally geodesic flat tori; their first-order rigidity
results give further context. Weaker curvature conditions also distinguish
this product from the round product geometry. For example,
Dom\'inguez-V\'azquez, Gonz\'alez-\'Alvaro, and
Mouill\'e~\cite[Theorem~E and Remark~1.2]{DominguezGonzalezMouille2023}
discuss a construction of Wilking on \(S^3\times S^3\) with positive
second intermediate Ricci curvature.
That condition requires positivity of certain sums of two sectional
curvatures, rather than positivity of every sectional curvature.
The flat-torus obstruction forces a nonnegative first curvature variation
to vanish on each totally geodesic flat torus, so a deformation of such a
background toward positive sectional curvature on $S^3\times S^3$ requires
control of higher-order terms.
Berger's first-order rigidity for product metrics is recalled in
\cite[Section~1.1]{BourguignonDeschampsSentenac1972}.
Bourguignon, Deschamps, and Sentenac developed a higher-order theory for
products without infinitesimal isometries
\cite[Section~7]{BourguignonDeschampsSentenac1972}; they explicitly explain
why their theorem does not extend to metrics with infinitesimal
isometries~\cite[Section~1.2]{BourguignonDeschampsSentenac1973}.
Strake's linearization of the Gauss equation
\cite[Lemma~4.1 and Proposition~4.3]{Strake1987} shows why a totally
geodesic flat torus cannot be made positively curved everywhere at first
order: a nonnegative first curvature variation must have zero integral
and hence vanish identically. The task is to control the higher-order
terms while preserving positivity on nearby planes.

Brendle and Hung~\cite{BrendleHung2026} address this difficulty on
\(S^2\times S^2\) by a third-order deformation of a Cheeger--M\"uter
metric. They first minimize curvature in directions normal to the
background zero set in the manifold of tangent two-planes. This
minimization incorporates the potentially negative contribution from moving
the plane as the metric changes. They then arrange a nonnegative
second-order coefficient which vanishes only on the tangent planes of one
flat torus. A third-order coefficient with nonzero average, followed by a
conformal correction obtained from a Poisson equation on the torus,
removes this last degeneracy. Their use of the full manifold of tangent
two-planes also treats points at which fixed-base plane coordinates become
degenerate. The higher-order strategy and this minimization framework are
the starting point of the present paper.
The construction on \(S^2\times S^3\) in~\cite{GuoFangLu2026} uses a
different feature of the geometry. It keeps a four-dimensional base
nonnegatively curved, passes to a circle bundle, and perturbs the connection
metric by a global symmetric tensor that rotates under the circle action.
Differentiation in the short fiber direction gives a positive second-order
contribution through the Gauss equation. An estimate distinguishing
horizontal and vertical plane variations shows that this contribution
dominates the loss from minimizing over nearby planes. This provides
positivity at second order. Our construction is directly on
\(S^3\times S^3\); it neither lifts that positive metric nor uses the
existence theorem on either sphere product as a hypothesis.
Our main result is the following.

\begin{theorem}\label{thm:main}
The smooth manifold \(S^3\times S^3\) admits a smooth Riemannian metric with
strictly positive sectional curvature.
\end{theorem}

The construction begins with two successive diagonal Cheeger deformations that leave two
distinct families of totally geodesic flat tori. One consists of products
of the same great circle; the other consists of products of great circles
in orthogonal two-planes. A single perturbation must work on both families.
Moreover, their projections to \(S^3\times S^3\) have different
degeneracies, so the calculation cannot be confined to plane variations at
a fixed base point.
Our construction separates these requirements. An invariant tensor improves
the second-order coefficient on the second family while having no
quadratic interaction with the other perturbations on the first.
Two height functions then select one torus in the first family and give
a quadratic lower bound away from it. Smaller perturbations are chosen to
produce the required positive cubic average on this torus. Their mixed
second-order terms must vanish there together with their first transverse
derivatives; mere vanishing of the values would not preserve a quadratic
lower bound. We realize this cancellation by a smooth global tensor,
including across the diagonal and antidiagonal. Finally, parity ensures
that this correction and the tensor used on the second family do not
alter the relevant cubic average. The conformal Poisson correction then
completes the argument.

\vspace{3pt}

\noindent\textbf{The role of AI in this proof.}
Odin Automatic AI Research Agent was used to construct the  metric and the proof.

\vspace{3pt}

\noindent {\bf Paper organization.}
Section~\ref{sec:reduction} fixes the notation and curvature conventions.
Section~\ref{sec:tensors} gives the complete metric construction, explains
the constraints imposed at each stage, and distinguishes free parameters
from coefficients determined by cancellation. The geometric and curvature
proofs follow in Sections~\ref{sec:background}--\ref{sec:positivity}.
The appendix records the finite algebraic identities,
an alternative algebraic mixed correction, and details of coefficient
selection. The code is available at
{\def\UrlBreaks{\do\/\do\-}\url{https://github.com/shengtaoguo/s3xs3-curvature-calculations}}.

%% file: reduction.tex
\section{Preliminaries and notation}\label{sec:reduction}

\subsection{Quaternionic coordinates and curvature conventions}
Throughout, \(M=S^3\times S^3\), and
\(\cN=\Gr_2(TM)\) is the compact manifold of unoriented tangent
two-planes. A point of \(\cN\) records both a base point and a plane.
We represent \(S^3\) by the unit quaternions and write
\(p=p_0+p_1i+p_2j+p_3k\) and \(q=q_0+q_1i+q_2j+q_3k\).
For an imaginary quaternion \(a\), put
\(E_a=(pa,0)\) and \(F_a=(0,qa)\). The combined left Maurer--Cartan
coframe is denoted by \((\alpha_1,\alpha_2,\alpha_3,
\beta_1,\beta_2,\beta_3)\). The corresponding right coframes are obtained
by applying the adjoint actions:
\(\widehat\alpha=\Ad_p\alpha\) and \(\widehat\beta=\Ad_q\beta\).
Thus \([E_a,E_b]=E_{[a,b]}\), with \([a,b]=2a\times b\), and similarly
on the second factor. We use
\(\lambda\odot\mu=(\lambda\otimes\mu+\mu\otimes\lambda)/2\).
Unless stated otherwise, matrices of covariant tensors are written in this
combined left coframe.

Set \(r=p^{-1}q=(r_0,v)\), \(R=\Ad_r\), and
\(\Delta_\pm=\{(p,q):q=\pm p\}\). The symbol \(I\) denotes the identity
on the indicated space; in the three-by-three tensor blocks it means
\(I_3\). We identify the
standard basis \(e_0,e_1,e_2,e_3\) of \(\R^4\) with
\(1,i,j,k\). For \(L\in\Gr_2(\R^4)\), write \(C_L=L\cap S^3\).
All tensor products and contractions below are on \(M\).

For a Riemannian metric \(g\) with Levi--Civita connection \(\nabla\),
we use the curvature endomorphism
\begin{equation}\label{eq:curvature-convention}
 R^g(X,Y)Z=\nabla_X\nabla_YZ-\nabla_Y\nabla_XZ-\nabla_{[X,Y]}Z.
\end{equation}
Thus the numerator of sectional curvature is
\(g(R^g(X,Y)Y,X)\).
All series below use Taylor coefficients, so the coefficient of order
\(j\) is the \(j\)-th derivative divided by \(j!\).

\subsection{Normal minimization in base and plane directions}
We use the minimum-curvature framework of Brendle and
Hung~\cite[Lemma~2.1 and Proposition~2.9]{BrendleHung2026}.
The following two statements specify the coefficients used in the
construction. Their proofs, including the normalization comparison, are
given in Section~\ref{sec:normal-proofs}.

\begin{lemma}[Normal minimum]\label{lem:normal-minimum}
Let \(N\) be a compact manifold, and let \(Z\subset N\) be a compact
embedded submanifold. Suppose a smooth family of functions has expansion
\(F_\tau=F_0+\tau F_1+\tau^2F_2+\tau^3F_3+O(\tau^4)\),
where \(F_0\geq0\), \(F_0^{-1}(0)=Z\), the Hessian of \(F_0\) is
positive definite on the normal bundle of \(Z\), and \(F_1|_Z=0\).
Fix tubular coordinates \((z,\nu)\) near \(Z\). For sufficiently small
\(|\tau|\), minimization in \(\nu\) gives a smooth function
\[
 F_{\min}(\tau,z)=\tau^2q(z)+\tau^3r(z)+O(\tau^4),
\]
uniformly on \(Z\). If \(A_z=D_\nu^2F_0(z,0)\) and
\(b_z=D_\nu F_1(z,0)\), then the first minimizing displacement and
second coefficient are
\begin{equation}\label{eq:schur-reduction}
 \eta_1=-A_z^{-1}b_z,\qquad
 q=F_2|_Z-\tfrac12 b_z^TA_z^{-1}b_z.
\end{equation}
The cubic coefficient is
\begin{equation}\label{eq:cubic-reduction}
 r=F_3+D_\nu F_2[\eta_1]
   +\tfrac12D_\nu^2F_1[\eta_1,\eta_1]
   +\tfrac16D_\nu^3F_0[\eta_1,\eta_1,\eta_1],
\end{equation}
with all quantities evaluated at \((z,0)\). If \(q\geq0\) on \(Z\) and
\(r>0\) on \(q^{-1}(0)\), then \(F_\tau>0\) on \(N\) for every
sufficiently small \(\tau>0\).
\end{lemma}

\begin{lemma}[Change of normal section]\label{lem:change-section}
Under the hypotheses of Lemma~\ref{lem:normal-minimum}, the second reduced
coefficient is independent of the tubular choice. Two first minimizing
displacements differ by a tangent vector to \(Z\). At a point where
\(q=0\) and \(\dd q=0\) along \(Z\), they give the same cubic
coefficient. Changing the second displacement also does not change that
coefficient.
\end{lemma}

For a nonnegatively curved background \(g_0\) whose zero set \(\cZ\) satisfies Lemma~\ref{lem:normal-minimum}, apply the lemma with \(N=\cN=\Gr_2(TM)\) and
\(F_\tau(\sigma)=\Sec_{g_0+\tau h+\tau^2k}(\sigma)\).
For a smooth symmetric tensor \(h\), write
\(F_{1,h}(\sigma)=[\tau]\Sec_{g_0+\tau h}(\sigma)\) on \(\cN\), and define \(\cL h=[\tau]\Sec_{g_0+\tau h}|_{\cZ}\).
When \(h\in\ker\cL\), write \(\cQ(h)\) for the second coefficient after normal minimization
with \(k=0\). This is quadratic in \(h\); we use the polarization
\(\cQ(h_1,h_2)=\tfrac12\bigl( \cQ(h_1+h_2)-\cQ(h_1)-\cQ(h_2)\bigr), \quad h_1,h_2\in\ker\cL\).
For the path with both tensors, the
normal minimum is
\begin{equation}\label{eq:metric-reduction}
 \tau^2\bigl(\cQ(h)+\cL k\bigr)
 +\tau^3\cR(h,k)+O(\tau^4).
\end{equation}
We fix one system of tubular coordinates when \(\cR(h,k)\) is used away from a
critical zero of the second coefficient.

The operators \(\cL,\cQ,\cR\) take values in scalar
functions on \(\cZ\). In particular,
\(\cQ\) includes the Schur term in \eqref{eq:schur-reduction}.
For the final tensor pair \(H,K\), we denote these coefficients by
\(\Vtwo\) and \(\Vthree\).

Our normalization differs from that used in \cite{BrendleHung2026}. The global curvature function in
\cite[Proposition~5.3]{BrendleHung2026} uses the fixed background Gram
determinant. Here \(F_\tau\) is sectional curvature and uses the perturbed
Gram determinant. For a fixed background and metric path, the normal
minima of the two normalized functions have the same reduced second
coefficient on \(\cZ\) when \(\cL h=0\), and the same reduced third coefficient where that
second coefficient vanishes. Independence of the normal coordinates also
requires its differential to vanish, as in Lemma~\ref{lem:change-section}.
For an unnormalized curvature numerator one must additionally divide by
the background determinant; the factors \(D_{\mathrm I}\) and
\(D_{\mathrm{II}}\) in Section~\ref{sec:tensors} make this conversion explicit.

\subsection{Linearization on a flat torus}
On a totally geodesic flat torus with constant induced metric in
coordinates \(s,t\), put
\(X=\partial_s\), \(Y=\partial_t\). These vectors need not have unit
\(g_0\)-length. The totally geodesic restriction principle
\cite[Lemma~4.1]{Strake1987} identifies the ambient first variation
with the intrinsic one. Linearizing the Gauss equation gives
\begin{equation}\label{eq:torus-linearization}
 \cL k=\frac1D\left(
 \partial_s\partial_t k(X,Y)
 -\tfrac12\partial_s^2k(Y,Y)
 -\tfrac12\partial_t^2k(X,X)\right),
\end{equation}
where \(D=g_0(X,X)g_0(Y,Y)-g_0(X,Y)^2\).
\begin{remark}\label{rem:first-order-obstruction}
This is Strake's flat-torus obstruction
\cite[Lemma~4.1 and Proposition~4.3]{Strake1987} in our notation.
If \(g_0+\tau h+O(\tau^2)\) has positive sectional curvature for all
small positive \(\tau\), then \(\cL h\geq0\) on each flat torus.
The first ambient variation there equals the intrinsic Gaussian-curvature
variation, because the initial second fundamental form vanishes.
The Gauss--Bonnet theorem gives \(\int_T\cL h\,\dd A_{g_0}=0\), so
\(\cL h=0\) everywhere on \(T\). Thus the vanishing first-variation condition used
in \eqref{eq:metric-reduction} is forced.
\end{remark}

%% file: tensors.tex
\section{Construction of the metric}\label{sec:tensors}

We construct the metric in the following order. Section~\ref{sec:background-choice}
chooses a Cheeger-deformed background with two families of flat tori, and
Section~\ref{sec:H0-choice} constructs a tensor that improves the second
family without affecting the first. Section~\ref{sec:height-choice}
uses height functions to obtain a quadratic lower bound on the first
family, with equality only along one torus. Section~\ref{sec:BC-choice}
adds the tensors needed for a positive cubic mean and identifies their
second-order cancellation requirements; Section~\ref{sec:K1-choice}
constructs the remaining global mixed correction.
Section~\ref{sec:amplitude-choice} fixes the amplitudes so that the
quadratic lower bound survives. Finally, Section~\ref{sec:complete-metric}
uses a Poisson equation to turn the cubic mean into a positive constant
and states the complete metric. The choice of the principal numerical
coefficients in the construction is explained in Appendix~\ref{app:coefficient-choice}.

Recall from Section~\ref{sec:reduction} that \(\cL\) is the first
curvature coefficient on the background zero set, whereas \(\cQ\) and
\(\cR\) are the second and third coefficients after normal minimization.
Thus the second-order objective is \(\cQ(H)+\cL K\geq0\), with
zero set the tangent-plane lift of one flat torus; the third-order
objective is a positive average on that torus. We specify the metric
data here and prove the curvature assertions in
Sections~\ref{sec:background}--\ref{sec:positivity}.

\subsection{Choosing the background and its zero set}\label{sec:background-choice}

We first want a nonnegatively curved metric whose zero planes form smooth
compact families and whose curvature grows quadratically in directions
normal to those families.
To obtain the standard normal homogeneous metric, give the three
\(SU(2)\) factors in its quotient presentation the same bi-invariant
scale. Equivalently, start with the product metric \(3I_6\) and
perform the diagonal right Cheeger deformation with parameter \(1/3\).
With the convention that the acting group has metric \(t^{-1}Q\), where
\(i,j,k\) are \(Q\)-orthonormal, the result is
\begin{equation}\label{eq:intermediate}
 g_1=\begin{pmatrix}2I&-I\\-I&2I\end{pmatrix}.
\end{equation}
For unit bi-invariant scales, the standard quotient metric on
\(SU(2)^3/\Delta SU(2)\) is \(g_1/3\). Thus equality of the
three scales fixes the relative coefficients, while the overall factor
three is a convenient normalization.
A further diagonal left Cheeger deformation, initially with free
parameter \(t>0\), gives the cometric
\begin{equation}\label{eq:background-family}
 g_{0,t}^{-1}=
 \begin{pmatrix}
 (\frac23+t)I&\frac13I+tR\\
 \frac13I+tR^T&(\frac23+t)I
 \end{pmatrix}.
\end{equation}
The background now has one free parameter.
We choose \(t\) so that, for the resulting background, the height tensors satisfy
\eqref{eq:K0-matching} and~\eqref{eq:height-positivity-test}, to be specified later.
The other required checks are the separation identities
\eqref{eq:H0-radical}--\eqref{eq:H0-minus-positive}, the second-order
cancellations \eqref{eq:BB-cancel} and~\eqref{eq:BC-j}--\eqref{eq:BC-k},
and the positive cubic mean \eqref{eq:cubic-mean}. A feasible choice, satisfying all these conditions, is \(t=3/10\).
Substitution \(t=3/10\) in \eqref{eq:background-family} gives
\begin{equation}\label{eq:background-cometric}
 g_0^{-1}=\frac1{30}
 \begin{pmatrix}29I&10I+9R\\10I+9R^T&29I\end{pmatrix}.
\end{equation}

The first deformation makes the two axes of a surviving mixed plane
parallel. Denote their common unit imaginary quaternion axis by \(u\).
The second deformation imposes
\(Ru=u\) or \(Ru=-u\). These two alternatives give
\begin{align}
 \cZ_{\mathrm I}
 &=\{T_{(p,q)}(C_L\times C_L):p,q\in C_L,
                                  \ L\in\Gr_2(\R^4)\},
 \label{eq:plus-incidence}\\
 \cZ_{\mathrm{II}}
 &=\{T_{(p,q)}(C_L\times C_{L^\perp}):p\in C_L,
                  \ q\in C_{L^\perp},\ L\in\Gr_2(\R^4)\}.
 \label{eq:minus-incidence}
\end{align}
These are subsets of \(\cN\): the base point is part of each element.
Write \(\cZ=\cZ_{\mathrm I}\sqcup\cZ_{\mathrm{II}}\) and
\(\cZ_{\mathrm I}^{\mathrm{reg}}=\{((p,q),\sigma)\in\cZ_{\mathrm I}:
 p\ne\pm q\}\). We refer to the corresponding tori as type I and
type II.
To obtain the induced matrices, restrict
\eqref{eq:background-cometric} to the two common-axis directions and
invert those two-dimensional blocks. In round-unit-speed circle
coordinates \(s,t\), this gives
\begin{equation}\label{eq:torus-matrices}
 G_{\mathrm I}=\frac1{16}
 \begin{pmatrix}29&-19\\-19&29\end{pmatrix},\quad
 G_{\mathrm{II}}=\frac1{28}
 \begin{pmatrix}29&-1\\-1&29\end{pmatrix},\qquad
 D_{\mathrm I}=\frac{15}{8},\quad D_{\mathrm{II}}=\frac{15}{14}.
\end{equation}

In circle coordinates the letter \(t\) is an angular
variable; the Cheeger parameter has already been fixed.
Propositions~\ref{prop:background} and~\ref{prop:clean-zero} prove that
these are totally geodesic flat tori, that the two six-dimensional zero
families are smooth and disjoint, and that the full normal Hessian has
rank $8$.
On type II, one normal direction changes
the base point away from \(\langle p,q\rangle=0\). On type I over
\(p=\pm q\), two such directions are needed. The normal minimization
in Section~\ref{sec:reduction} includes these base directions together
with the plane directions.

\subsection{Separating the two families}\label{sec:H0-choice}

A large perturbation intended for one zero family could destroy the
second-order estimate on the other. We therefore seek a tensor \(H_0\)
with \(\cL H_0=0\) and the stronger properties
\[
 \cQ(H_0,h)=0\quad\text{on }\cZ_{\mathrm I}
       \text{ for every }h\in\ker\cL,
 \qquad \cQ(H_0)>0\quad\text{on }\cZ_{\mathrm{II}}.
\]
For \(P(v)=|v|^2I-vv^T\), set
\[
 H_0=r_0\begin{pmatrix}0&P(v)\\P(v)&0\end{pmatrix}.
\]
The formula reflects the two alternatives above. On type I, \(v\) is
parallel to the common axis, which \(P(v)\) annihilates. The invariant
metric path \(g_0+\tau H_0\) consequently leaves these tori flat and
totally geodesic. On type II, \(r_0=0\), but its derivative normal to
that hypersurface is nonzero. This produces the positive quadratic
contribution. Proposition~\ref{prop:H0-opening} proves
\eqref{eq:H0-radical} on type~I.

\subsection{Selecting one torus by height functions}\label{sec:height-choice}

Set \(L_0=\Span\{e_0,e_1\}\) and
\(T_0=C_{L_0}\times C_{L_0}\). Its tangent-plane lift is
\(\Tlift=\iota_0(T_0)\), where
\(\iota_0(p,q)=((p,q),T_{(p,q)}T_0)\).
For \(L\in\Gr_2(\R^4)\), let \(e_\perp,f_\perp\) be the orthogonal
projections of \(e_0,e_1\) to \(L^\perp\), and put
\(\rho(L)=|e_\perp|^2+|f_\perp|^2\).
This vanishes exactly at \(L_0\) and measures squared transverse distance
to \(\Tlift\) on \(\cZ_{\mathrm I}\). Using two orthogonal height
vectors selects this supporting plane.

For the first-order height tensor, take the ansatz
\[
 p_0\begin{pmatrix}I&a_0I\\a_0I&0\end{pmatrix}
 +q_1\begin{pmatrix}0&a_0I\\a_0I&I\end{pmatrix}.
\]
Its tangential entries on a flat torus are \(p_0(s),q_1(t)\), and
\(a_0(p_0(s)+q_1(t))\), so its first curvature variation is zero for every
\(a_0\). The value \(a_0=-2/7\), together with the correction below,
satisfies the cancellation equation \eqref{eq:K0-free-equation} and the
transverse positivity test \eqref{eq:height-positivity-test}. A reproducible
selection procedure is given in Appendix~\ref{app:coefficient-choice}.
The resulting tensor is
\begin{equation}\label{eq:odd-tensor}
 O=p_0\begin{pmatrix}I&-\frac27I\\-\frac27I&0\end{pmatrix}
   +q_1\begin{pmatrix}0&-\frac27I\\-\frac27I&I\end{pmatrix}.
\end{equation}
Write \(O=O_p+O_q\) for its two summands.

The correction must cancel the curvature on \(T_0\) and leave a
positive multiple of \(\rho\) elsewhere. Use the two-dimensional ansatz
\[
 K_0(\zeta_0,\zeta_1)
 =\zeta_0p_0q_1\begin{pmatrix}0&I\\I&0\end{pmatrix}
                          +\zeta_1p_1q_0I_6.
\]
On \(T_0\), write
\(p=e^{si}\), \(q=e^{ti}\). Using \(a_0=-2/7\) in
\eqref{eq:height-free-coefficient} and~\eqref{eq:K0-free-equation} gives
\begin{equation}\label{eq:K0-matching}
\begin{gathered}
 \cQ(O)=-\frac{389}{11025}\sin s\cos t,\quad
 \cL K_0(\zeta_0,\zeta_1)=\frac8{15}(\zeta_1-\zeta_0)\sin s\cos t, \quad
 \zeta_1-\zeta_0=\frac{389}{5880}.
\end{gathered}
\end{equation}
The last equation is the required cancellation in \eqref{eq:K0-free-equation}.
There remains one free coefficient. We choose \(\zeta_0=-1/5\) to satisfy
the transverse positivity test \eqref{eq:height-positivity-test} below.
This gives
\begin{equation}\label{eq:K0}
 K_0=-\frac15p_0q_1\begin{pmatrix}0&I\\I&0\end{pmatrix}
       -\frac{787}{5880}p_1q_0I_6.
\end{equation}

The remaining test is positivity transverse to \(\Tlift\).
For these data, Proposition~\ref{prop:odd-margin} proves
\(\cQ(O)+\cL K_0\geq\rho/250\) on \(\cZ_{\mathrm I}\),
as stated in \eqref{eq:odd-lower-bound}.
Here is the finite selection problem behind this test. After normalizing
\(p=1,q=\cos\theta+i\sin\theta\), the curvature expression is quadratic
in the transformed orthonormal height vectors \(e,f\).
The matrices in \eqref{eq:odd-coefficient-matrices} retain the term
\(d_0\langle e,f\rangle\), which vanishes since \(e\perp f\).
After this cancellation, the expression depends only on the transverse
components \(e_\perp,f_\perp\). The resulting quadratic form on
\(L^\perp\oplus L^\perp\) must be positive definite for every
\(x=\sin^2\theta\in[0,1]\).
For the polynomials \(P,V,W,d\) defined in
\eqref{eq:height-polynomials}, a sufficient scalar test is
\begin{equation}\label{eq:height-positivity-test}
 d(x)>0,\qquad
 \frac{P(x)-\max\{|V(x)|,|W(x)|\}}{2d(x)}>\frac1{250}
                         \quad(0\leq x\leq1).
\end{equation}
The bounds \eqref{eq:height-polynomial-bounds} and the exact rational
inequality \eqref{eq:height-numerical-margin} prove this test and
\eqref{eq:odd-lower-bound}. The rational values \(t =3/10\),
\(a_0=-2/7\), and \(\zeta_0=-1/5\) therefore satisfy
\eqref{eq:height-positivity-test}.
Appendix~\ref{app:coefficient-choice} describes the coefficient matching
and the remaining freedom more explicitly.

\subsection{Producing a cubic term without losing the quadratic bound}
\label{sec:BC-choice}

The height construction alone leaves \(\Tlift\). We need another
first-order tensor whose minimizing plane has a nonzero first displacement,
and a further tensor with which it has a useful cubic interaction.
The left and right Maurer--Cartan forms are global and have constant
restrictions to the great-circle tori. They therefore provide an ansatz
that respects the necessary condition \(\cL h=0\). Set
\begin{equation}\label{eq:B-tensor}
 B=\tfrac12(\alpha_1+\widehat\alpha_1)
       \odot\left(\alpha_2+\tfrac3{13}\beta_2\right)
   -\tfrac12(\beta_1+\widehat\beta_1)
       \odot\left(\tfrac3{13}\alpha_2+\beta_2\right),
\end{equation}
and set \(C=\alpha_1^2\).
The tensor \(B\) moves the minimizing plane to first order on
\(\Tlift\), whereas \(C\) has zero first normal displacement there.
These assertions are proved in Lemma~\ref{lem:pole-free}.
The factor \(1/2\) fixes a normalization of \(B\); its amplitude is
\(\eps\) in \eqref{eq:small-tensor-choice}. Replacing \(B\) by
\(uB\) would require replacing \(K_1,K_2,K_3\) by
\(uK_1,u^2K_2,uK_3\), respectively, as follows from the
homogeneities in \eqref{eq:five-corrections}. Its relative coefficient has a more specific role.
Replace \(3/13\) in \eqref{eq:B-tensor} by a free number \(b_B\), and
seek a correction proportional to \((2r_0^2-1)g_0\). On \(T_0\),
this correction produces only the Fourier mode \(\cos2(t-s)\).
Consequently, we require \(\cQ(B_{b_B})|_{\Tlift}\) to be a multiple of
that mode. Comparing its values at
\((\cos\theta,\sin\theta)=(0,1)\) and \((3/5,4/5)\) already forces
\((13b_B-3)^2=0\), hence \(b_B=3/13\).
Equations~\eqref{eq:B-selection-values} and~\eqref{eq:B-selection-equation}
give this exact two-point calculation. The full identity
\eqref{eq:BB-cancel} proves that the selected value works for every
\(\theta\). The remaining transverse compatibility and positive cubic
mean follow from \eqref{eq:BC-j}--\eqref{eq:BC-k} and the positive
constant term in the cubic formula~\eqref{eq:basic-cubic}.

A crucial compatibility condition is easily missed. A perturbation of a
quadratic bound \(c|U|^2\) by a linear term \(\eps\ell(U)\) can be
negative arbitrarily close to \(U=0\), however small \(\eps>0\) is.
Consequently, we require each additional corrected second-order term to
vanish on \(\Tlift\) together with its first derivative along
\(\cZ_{\mathrm I}\). The five terms to be controlled are
\begin{equation}\label{eq:five-corrections}
\begin{gathered}
 2\cQ(O,B)+\cL K_1,\quad 2\cQ(O,C), \quad
 \cQ(B)+\cL K_2,\quad
 2\cQ(B,C)+\cL K_3, \text{ and } \cQ(C).
\end{gathered}
\end{equation}

For \(K_2\), the calculation on \(T_0\) gives
\(\cQ(B)=(128/845)\cos2(t-s)\). The induced Laplacian is
\begin{equation}\label{eq:torus-laplacian}
 \Delta_T=\frac{29}{30}(\partial_s^2+\partial_t^2)
                    +\frac{38}{30}\partial_s\partial_t,
 \qquad \Delta_T\cos2(t-s)=-\frac83\cos2(t-s).
\end{equation}
Since \(\cL(fg_0)=-\frac12\Delta_T f\) on a background flat torus,
the ansatz \(\kappa_2(2r_0^2-1)g_0\) must satisfy
\begin{equation}\label{eq:K2-matching}
 \frac{128}{845}+\frac43\kappa_2=0,
 \qquad \kappa_2=-\frac{96}{845}.
\end{equation}
We therefore set \(K_2=-(96/845)(2r_0^2-1)g_0\).
For the mixed \(B,C\) term, parity already gives zero value on
\(\Tlift\); its transverse derivative must still be cancelled.
For the ansatz \(\kappa_3(Bg_0^{-1}C+Cg_0^{-1}B)\), the two transverse
identities \eqref{eq:BC-j} and~\eqref{eq:BC-k} give the same scalar equation:
\begin{equation}\label{eq:K3-matching}
 \frac{128}{325}-\frac{16}{65}\kappa_3=0,
 \qquad \kappa_3=\frac85.
\end{equation}
Consequently, we set
\(K_3=(8/5)(Bg_0^{-1}C+Cg_0^{-1}B)\).
The products denote contraction through the background cometric.
For the even functions \(2\cQ(O,C)\), \(\cQ(B)+\cL K_2\),
and \(\cQ(C)\), the values vanish by \eqref{eq:BB-cancel};
reflection across \(L_0\) forces their first transverse derivatives
to vanish. Proposition~\ref{prop:mixed-jets} verifies
all these conditions; the \(O,B\) correction is specified next.

\subsection{A global correction for the mixed term}\label{sec:K1-choice}

The remaining task is to cancel the value and first transverse derivative
of \(2\cQ(O,B)\) on \(\Tlift\). We also impose parities that will
preserve the cubic average. The reflection
\begin{equation}\label{eq:reflection}
 (p,q)\longmapsto(ipi^{-1},iqi^{-1})
\end{equation}
fixes \(\Tlift\) pointwise and reverses its four transverse directions
in \(\cZ_{\mathrm I}\); \(O,C,H_0\) are even and \(B\) is odd.
We require \(K_1\) to be odd under this reflection and under
\((p,q)\mapsto(-p,-q)\), and to satisfy
\begin{equation}\label{eq:admissible-mixed}
 S:=2\cQ(O,B)+\cL K_1,
 \qquad S|_{\Tlift}=0,\qquad \dd S|_{\Tlift}=0.
\end{equation}
The differential is taken on \(\cZ_{\mathrm I}\).
Since \(O,B\) are fixed, these are linear conditions on \(K_1\)
along \(\Tlift\).

Here is a complete geometric choice. Put \(N=L_0^\perp\),
\(p_s=\cos s\,e_0+\sin s\,e_1\), and
\(q_t=\cos t\,e_0+\sin t\,e_1\). Parametrize nearby supporting
planes by graphs \(U:L_0\to N\). With
\(\iota:L_0\hookrightarrow\R^4\) the inclusion, set
\[
 A_U=(\iota+U)(I_{L_0}+U^TU)^{-1/2},\qquad
 (p,q)=(A_Up_s,A_Uq_t).
\]
Define the linear functionals \(L_p(s,t)[U]\) and \(L_q(s,t)[U]\)
to be the derivatives at \(U=0\) of
\(2\cQ(O_p,B)\) and \(2\cQ(O_q,B)\), respectively, evaluated on
the tangent planes of these graph tori.
The metric and tensors above determine these functionals.
The functional \(L_p\) is
\(\pi\)-antiperiodic in \(s\) and \(\pi\)-periodic in \(t\),
with the reverse parities for \(L_q\).

To cancel these derivatives through \eqref{eq:torus-linearization},
we need second primitives in the indicated angular variables. Set
\begin{equation}\label{eq:mixed-primitives}
\begin{aligned}
 \kappa_p(s,t)[U]
 &=D_{\mathrm I}\int_0^\pi(\xi-\pi/2)
                         L_p(s+\xi,t)[U]\,\dd\xi,\\
 \kappa_q(s,t)[U]
 &=D_{\mathrm I}\int_0^\pi(\xi-\pi/2)
                         L_q(s,t+\xi)[U]\,\dd\xi.
\end{aligned}
\end{equation}
Antiperiodicity gives \eqref{eq:mixed-primitive-equations}:
\(\partial_s^2\kappa_p=2D_{\mathrm I}L_p\) and
\(\partial_t^2\kappa_q=2D_{\mathrm I}L_q\).
The factor \(2D_{\mathrm I}=15/4\) cancels the coefficient
\(-1/(2D_{\mathrm I})\) of the corresponding second derivative
in \eqref{eq:torus-linearization}.
 Represent these
functionals by
\(\kappa_p[U]=\langle D_p,U\rangle_{\mathrm{HS}}\) and
\(\kappa_q[U]=\langle D_q,U\rangle_{\mathrm{HS}}\), where
\(D_p,D_q:L_0\to N\). The subscript \(\mathrm{HS}\) denotes the
Euclidean Hilbert--Schmidt inner product.

In a tubular neighborhood of \(T_0\), write
\(p=\sqrt{1-|x|^2}\,p_s+x,\quad q=\sqrt{1-|y|^2}\,q_t+y,\quad x,y\in N\).
Choose any smooth radial cutoff \(\chi(|x|^2,|y|^2)\) equal to one
near zero and supported in this neighborhood, and define
\begin{equation}\label{eq:geometric-K1}
\begin{aligned}
 k_p&=\chi\bigl[\langle D_pq_t,y\rangle\,\dd t^2
             +\dd t\odot\langle D_pq_t',\dd y\rangle\bigr],\\
 k_q&=\chi\bigl[\langle D_qp_s,x\rangle\,\dd s^2
             +\dd s\odot\langle D_qp_s',\dd x\rangle\bigr],\\
 K_1&=k_p+k_q.
\end{aligned}
\end{equation}
Extend by zero outside the neighborhood. The use of both \(q_t\) and
\(q_t'\) is essential: they form an orthonormal basis of \(L_0\).
At \(p=\pm q\), a transverse variation can have \(Uq_t=0\) but
\(Uq_t'\ne0\); the tangent--normal component detects the latter.
Thus the construction does not divide by \(\sin(t-s)\); it remains
smooth across the diagonal and antidiagonal and realizes the prescribed
first-order transverse jet.
Proposition~\ref{prop:mixed-existence} proves smoothness, parity, and
\eqref{eq:admissible-mixed}.

\subsection{Fixing the amplitudes}\label{sec:amplitude-choice}

The first-derivative cancellations are precisely what makes a uniform
small choice possible. Define functions on \(\cZ_{\mathrm I}\) by
\begin{equation}\label{eq:FG}
\begin{aligned}
 F&=2\cQ(O,B)+\cL K_1+2\cQ(O,C),\\
 G&=\cQ(B)+\cL K_2+2\cQ(B,C)+\cL K_3+\cQ(C).
\end{aligned}
\end{equation}
Proposition~\ref{prop:mixed-jets} gives \(F,G=O(\rho)\).
Consequently the following constant is finite:
\begin{equation}\label{eq:Cstar}
 C_*=1+\sup_{\cZ_{\mathrm I}\setminus\Tlift}
                              \frac{|F|+|G|}{\rho}.
\end{equation}
Choose, in this order,
\begin{equation}\label{eq:small-tensor-choice}
 \eps=\frac1{1000C_*},\qquad h_*=O+\eps(B+C),\qquad
 K=K_0+\eps K_1+\eps^2(K_2+K_3).
\end{equation}
Indeed,
\(\cQ(h_*)+\cL K=\cQ(O)+\cL K_0+\eps F+\eps^2G\), so
\(\eps\leq1\) and \(\eps C_*\leq1/1000\) imply~\eqref{eq:plus-margin}.

We now make the coefficient positive on type II without changing it on
type I. Define
\begin{align}
 M_*&=1+\|2\cQ(H_0,h_*)\|_{\infty,\cZ_{\mathrm{II}}}
       +\|\cQ(h_*)+\cL K\|_{\infty,\cZ_{\mathrm{II}}},
       \label{eq:Mstar}\\
 c_*&=20M_*,\qquad H=h_*+c_*H_0.\label{eq:H-final}
\end{align}
These suprema concern smooth functions on a compact set.
Using \eqref{eq:H0-minus-positive} and \(M_*\geq1\), we obtain
\eqref{eq:typeII-amplitude-bound}:
\(\cQ(H)+\cL K\geq c_*^2/10-c_*M_*-M_*=20M_*^2-M_*>0\)
on \(\cZ_{\mathrm{II}}\).
The choice \(c_*=20M_*\) is therefore non-sharp; the displayed
quadratic inequality is the requirement it satisfies.
Proposition~\ref{prop:quadratic-final} proves \eqref{eq:q2}, namely
\(\Vtwo:=\cQ(H)+\cL K\geq0\) with zero set \(\Tlift\), and
also \(\dd\Vtwo|_{\Tlift}=0\).
The order of choices matters: fix \(K_1\), then \(\eps\), then \(c_*\).
Only after all of them have been fixed will \(\tau\) tend to zero.

\subsection{The cubic mean and the complete metric}\label{sec:complete-metric}

For the fixed \(H,K\), let \(\Vthree=\cR(H,K)\), and write
\(\Vthreebase=\iota_0^*\Vthree\) as a function of \(s,t\) on \(T_0\).
The zero value and differential of \(\Vtwo\) make this restriction
independent of the normal section.

By \eqref{eq:cubic-mean}, we have
\(m:=\mean_{T_0}\Vthreebase=(512/375)\eps^3>0\), where the mean uses the background area form normalized to mass one.
The mechanism is a cubic interaction with two copies of \(B\) and one
of \(C\). Reflection removes odd powers of \(B\); the Gauss--Bonnet theorem
removes the mean when \(B\) is absent; antipodal symmetries remove the
remaining terms involving \(O\) or \(H_0\). Hence neither the large
coefficient \(c_*\) nor the freedom in \(K_1\) destroys this positive
mean. The fraction \(512/375\) is the constant Fourier coefficient in
\eqref{eq:basic-cubic}; the oscillatory term has zero mean.

Following \cite[Section~6]{BrendleHung2026}, let \(f\) be the unique smooth mean-zero
solution of
\begin{equation}\label{eq:poisson}
 \Delta_T f=2(\Vthreebase-m),
\end{equation}
with \(\Delta_T\) given in \eqref{eq:torus-laplacian}.
The factor \(2\) cancels the conformal first-variation factor
\(-1/2\) in \eqref{eq:conformal-linearization}.
The zero-mean right-hand side is exactly the solvability condition;
the positivity of \(m\) makes the resulting constant curvature coefficient
useful.
 For a concrete extension, put
\(a_p=p_0^2+p_1^2\), \(a_q=q_0^2+q_1^2\), and, where defined,
\(z_p=(p_0+ip_1)/\sqrt{a_p}\) and
\(z_q=(q_0+iq_1)/\sqrt{a_q}\).

Choose a smooth cutoff \(\chi\) equal to zero on
\((-\infty,1/2]\) and one on \([3/4,\infty)\), regard \(f\) as a
function on \(S^1\times S^1\), and set
\begin{equation}\label{eq:smooth-extension}
 \widetilde f(p,q)=\chi(a_p)\chi(a_q)f(z_p,z_q),
\end{equation}
extended by zero wherever a projection is undefined.
This is a new cutoff, unrelated to the one in \eqref{eq:geometric-K1}.
Its thresholds only keep the support away from the undefined projections;
any fixed thresholds strictly between zero and one, in increasing order,
would work.

\begin{theorem}[Perturbation family]\label{thm:construction}
Fix the background and tensors specified in this section. Choose
\(K_1\) by \eqref{eq:geometric-K1}, or by \eqref{eq:K1}, or more
generally any smooth symmetric \(2\)-tensor satisfying \eqref{eq:admissible-mixed}
and its stated parities. Fix the amplitudes by
\eqref{eq:Cstar}--\eqref{eq:H-final} and the function by
\eqref{eq:poisson}--\eqref{eq:smooth-extension}. Then there exists
\(\tau_0>0\) such that
\begin{equation}\label{eq:final-metric}
 g_\tau=g_0+\tau\big(c_*H_0+O+\eps(B+C)\big)
       +\tau^2\big(K_0+\eps K_1+\eps^2(K_2+K_3)\big)
       +\tau^3\widetilde f\,g_0
\end{equation}
is a smooth Riemannian metric with positive sectional curvature for every
\(0<\tau<\tau_0\).
\end{theorem}

Sections~\ref{sec:background}--\ref{sec:cubic} prove the background,
second-order, and cubic assertions used in this recipe.
Section~\ref{sec:positivity} proves the theorem by applying
Lemma~\ref{lem:normal-minimum}: the conformal term leaves \(\Vtwo\)
unchanged and makes the reduced cubic coefficient equal to \(m\) on
its zero set. Compactness gives a single \(\tau_0\) for all base
points and all two-planes.
The uniform remainder estimate \eqref{eq:final-normal-minimum} and
Lemma~\ref{lem:normal-minimum} determine a sufficiently small interval;
no explicit or maximal value of \(\tau_0\) is needed.

%% file: background.tex
\section{Geometry of the background and curvature expansions}\label{sec:background}

The metric and zero families in this section are those specified in
Section~\ref{sec:background-choice}. We first verify their geometric
properties, then justify the curvature expansions used in the construction.

\subsection{The Cheeger deformations and the zero-curvature tori}
\begin{proposition}\label{prop:background}
The metric \(g_0\) is smooth, positive definite, and nonnegatively curved.
Its zero-curvature planes are exactly
\begin{equation}\label{eq:zero-criterion}
 P_u=\Span\{E_u,F_u\},\qquad |u|=1,
 \qquad Ru=u\ \text{or}\ Ru=-u.
\end{equation}
\end{proposition}

\begin{proof}
We apply Cheeger's submersion construction
\cite{Cheeger1973,ONeill1966} in the metric-operator form of
M\"uter's formulas~\cite[Propositions~1.1 and~1.3]{Ziller2009},
successively to the two diagonal actions.
On each acting copy of \(SU(2)\), fix the bi-invariant metric \(Q\)
for which \(i,j,k\) are orthonormal. Our Cheeger parameter convention
is the quotient of \((M,g)\times(SU(2),t^{-1}Q)\) for \(t>0\).

Begin with the bi-invariant product metric \(3I_6\). A Cheeger
deformation by simultaneous right multiplication, with parameter \(1/3\),
gives \(g_1\) in \eqref{eq:intermediate}.
A second deformation, now by simultaneous left multiplication and with
parameter \(3/10\), adds
\(\frac3{10}\left(\begin{smallmatrix}I&R\\R^T&I\end{smallmatrix}\right)\)
to the cometric. This gives \eqref{eq:background-cometric}. These assertions
follow either from the defining submersion or from the metric-operator
formula in~\cite[Proposition~1.1]{Ziller2009}. Nonnegative curvature follows
from the submersion formula. The singular values of \(10I+9R\) are at most
\(19\), so the cometric eigenvalues lie in \([1/3,8/5]\); in particular
\(g_0\geq(5/8)I_6\).

A zero-curvature plane for the product metric is mixed. In the reparametrized Cheeger formula
\cite[Proposition~1.3]{Ziller2009}, the first group-bracket term forces its
two axes to be parallel. The remaining planes are precisely \(P_u\).
They are flat for \(g_1\): the corresponding abelian two-dimensional
subgroups, and their left translates, are totally geodesic by the Koszul
formula. For the second deformation, the two Lie algebra vectors in the
group-bracket term are \(2\Ad_pu-\Ad_qu\) and
\(2\Ad_qu-\Ad_pu\). After conjugation by \(p^{-1}\), their bracket is
a nonzero multiple of
\((2u-Ru)\times(-u+2Ru)=3u\times Ru\).
Every summand of the reparametrized curvature formula is nonnegative, so a
zero-curvature plane must satisfy \(Ru=\pm u\). At these planes the
reparametrization preserves \(P_u\): it rescales an orbit direction in
that plane and fixes its horizontal direction. The flat tori in Proposition~\ref{prop:clean-zero} prove that each plane satisfying this condition is indeed flat.
\end{proof}

\begin{proposition}\label{prop:clean-zero}
The zero set of the sectional-curvature function on \(\cN\) is
\(\cZ=\cZ_{\mathrm I}\sqcup\cZ_{\mathrm{II}}\). Its two components are
disjoint compact smooth six-dimensional submanifolds. Each torus in
\eqref{eq:plus-incidence}--\eqref{eq:minus-incidence} is flat and totally
geodesic. The sectional-curvature Hessian is positive definite on the
eight-dimensional normal bundle of \(\cZ\). Fix a smooth auxiliary
Riemannian metric on \(\cN\), and let \(\dist\) denote its induced
distance. There is \(b_0>0\), depending on this choice, such that
\begin{equation}\label{eq:clean-lower-bound}
 \Sec_{g_0}(\sigma)\geq b_0\dist(\sigma,\cZ)^2
 \qquad\text{for all }\sigma\in\Gr_2(TM).
\end{equation}
\end{proposition}

\begin{proof}
If \(Ru=u\), then \(r\in\Span\{1,u\}\), and both \(p,q\) lie in
the plane \(L=\Span\{p,pu\}\). If \(Ru=-u\), then \(r_0=0\) and
\(v\perp u\); hence \(q,qu\in L^\perp\). This identifies
\eqref{eq:zero-criterion} with the two incidence descriptions.
The point and lifted plane recover \(L\) smoothly: the projection of
\(P_u\) to the first tangent factor has constant rank one, and adjoining
the vector \(p\) recovers \(L\). Thus these compact incidence
bundles embed in \(\Gr_2(TM)\). Each has base dimension four and fiber
dimension two. The signs in \eqref{eq:zero-criterion} are mutually
exclusive, so the components are disjoint.

The metric is invariant under the diagonal \(SO(4)\)-action and under
either factor's antipodal map. The orthogonal map equal to \(+I\) on
\(L\) and \(-I\) on \(L^\perp\) belongs to \(SO(4)\); its diagonal
action fixes \(C_L\times C_L\). Composing its action on the second factor
with the antipodal map fixes \(C_L\times C_{L^\perp}\).
These fixed sets are totally geodesic.
Restricting the cometric
\eqref{eq:background-cometric} to the common-axis directions and inverting
the two-dimensional blocks gives the constant matrices
\(G_{\mathrm I},G_{\mathrm{II}}\) in
\eqref{eq:torus-matrices}, with the stated determinants.
 Constancy
proves flatness. Reversing one great-circle coordinate changes the sign
of the off-diagonal entry but not the determinant.

For the Hessian calculation we use the nonnegative terms in
M\"uter's formula~\cite[Proposition~1.3]{Ziller2009}; compare the
one-deformation estimate in \cite[Appendix~B]{BrendleHung2026}.
Here we track both bracket differentials in the full base-and-plane
bundle. All terms must be compared in the same variables. Let \(C_t\) be the positive metric operator defined by
\(g_t(V,W)=g_{\mathrm{old}}(C_tV,W)\). For a Cheeger deformation, the map
\(\Psi_t:\Gr_2(TM)\longrightarrow\Gr_2(TM),\quad (x,\sigma)\longmapsto(x,C_t^{-1}\sigma)\)
is a smooth diffeomorphism covering the identity of \(M\). Pull back the
second curvature formula by the composition of the two such maps, and
substitute the first curvature formula into its old-curvature term.
Locally choosing a smooth basis of each plane gives a sum of nonnegative
terms, up to positive smooth Gram factors. At a zero of a squared norm
\(\|B\|^2\), its Hessian is \(2(\dd B)^*\dd B\); positive factors do
not change its kernel. We check each new bracket differential on the
common kernel of the preceding terms.

For the product metric, a mixed plane has a basis \((a,0),(0,b)\).
Its graph variations have four components that move these two axes
within their own factors, and four that mix the factors. The latter
give the positive quadratic terms
\(\|a\times\delta Y_1\|^2+\|\delta X_2\times b\|^2\), up to positive
constants. Thus the product Hessian has rank four, with kernel the
tangent space to the ten-dimensional mixed-plane locus.
On that locus the first group bracket is a nonzero multiple of
\(a\times b\). At parallel axes its differential has rank two in the
relative-axis directions. Its kernel consists of the common-axis
variations and all base-point directions, giving the eight-dimensional
zero locus of \(g_1\), parametrized by \((p,q,[u])\) and the plane
\(P_u\). The first reparametrization fixes each such plane, so these
are also coordinates for the remaining common kernel.

On this kernel the second group bracket has the same differential rank
as \(\Phi(r,u)=u\times\Ad_r u\). At a type-I zero
\(r=(a,bu)\), \(a^2+b^2=1\), varying \(v\) by \(w\perp u\) gives
\[
 D(u\times Ru)[w]=2(aI+bJ_u)w,\qquad J_uw=u\times w.
\]
This map is invertible on \(u^\perp\), including when \(r=\pm1\).
At a type-II zero, variations of \(r_0\) and \(v\cdot u\) give the
independent images \(2v\) and \(2u\times v\). The second bracket thus
has rank two on the remaining kernel everywhere. The successive common
kernels have dimensions \(10,8,6\), respectively. The known six-dimensional
zero locus is contained in the kernel of the full curvature Hessian, while
these nonnegative terms already give rank eight. Hence the full kernel
is exactly \(T\cZ\), and the Hessian is positive definite on a normal
complement. Transport by the two diffeomorphisms above gives the same
conclusion in the original variables.

Taylor's theorem gives a uniform quadratic lower bound on a tube about
the compact zero set. On the compact complement, curvature has a positive
minimum. Decreasing the constant gives \eqref{eq:clean-lower-bound} on the
whole Grassmann bundle.
\end{proof}

The intermediate metric \eqref{eq:intermediate} also has the quotient
description in the introduction. The quotient of the product bi-invariant
metric on \(SU(2)^3\) has, in the coordinates
\([g_1,g_2,g_3]\mapsto(g_1g_3^{-1},g_2g_3^{-1})\), matrix
\(\frac13\left(\begin{smallmatrix}2I&-I\\-I&2I\end{smallmatrix}\right)\).
Thus \(g_1/3\) is precisely this quotient metric.

The ranks of the curvature Hessian restricted to the eight fixed-base
plane directions are
\[
\begin{array}{c|c|c}
 \text{zero-plane locus}&\text{fixed-base rank}
     &\text{additional base directions}\\ \hline
 \cZ_{\mathrm I}^{\mathrm{reg}}&8&0\\
 \cZ_{\mathrm I}\text{ over }(\Delta_+\cup\Delta_-)&6&2\\
 \cZ_{\mathrm{II}}&7&1
\end{array}
\]
Indeed, after the six directions supplied by the first deformation,
the remaining fixed-base variations are \(u\mapsto u+w\), \(w\perp u\).
For \(\Phi(u)=u\times Ru\), their differential is
\[
 D\Phi(w)=
 \begin{cases}
  J_u(R-I)w,&Ru=u,\\
  J_u(R+I)w,&Ru=-u.
 \end{cases}
\]
The first map has rank two off \(\Delta_+\cup\Delta_-\) and rank zero
on it; the second has rank one, since \(R\) is a half-turn. The base
variations used in Proposition~\ref{prop:clean-zero} supply the missing
directions. Appendix~\ref{app:normal-complements} gives explicit transverse
slices at the exceptional representatives.

\subsection{The normal-minimum formulas}\label{sec:normal-proofs}

\begin{proof}[Proof of Lemma~\ref{lem:normal-minimum}]
We use the implicit-function argument of
\cite[Lemma~2.1]{BrendleHung2026}, uniformly over the compact zero set.
Uniform positivity of the normal Hessian and the implicit function theorem
give a smooth critical section \(\nu_\tau(z)\), with
\(\nu_\tau=\tau\eta_1+O(\tau^2)\). On a sufficiently small fixed tube
the normal Hessian remains positive for small \(\tau\), so this critical
section is the normal minimum. Differentiating the critical equation gives
\(A_z\eta_1+b_z=0\). Substitution gives
\eqref{eq:schur-reduction}. In the cubic expansion, every term involving
the second displacement has the factor \(A_z\eta_1+b_z\); these terms
cancel, leaving \eqref{eq:cubic-reduction}. Compactness makes the
remainders uniform.

If \(q^{-1}(0)\) is nonempty, choose a neighborhood of it in \(Z\) on
which \(r\) has a positive lower bound. There the cubic term dominates the
fourth-order remainder. On the compact remainder of \(Z\), the positive
minimum of \(q\) dominates the higher terms. The same conclusion holds
immediately if \(q^{-1}(0)\) is empty. Positivity of each normal minimum
then gives positivity on the tube. On the compact complement, the positive
minimum of \(F_0\) persists.
\end{proof}

\begin{proof}[Proof of Lemma~\ref{lem:change-section}]
We use the coordinate-invariance principle of
\cite[Proposition~2.9]{BrendleHung2026}, and also record the effect
of the varying Gram determinant.
The Hessian of \(F_0\) has kernel \(TZ\), and \(\dd F_1\) annihilates
\(TZ\). Thus \eqref{eq:schur-reduction} defines a unique class of first
displacements modulo \(TZ\), and its value \(q\) is independent of the
chosen complement. If a second choice adds \(t\in T_zZ\), its footpoint
is \(z+\tau t+O(\tau^2)\). The reduced value is therefore
\(\tau^2q(z)+\tau^3\bigl(r(z)+\dd q_z[t]\bigr)+O(\tau^4)\).
This proves the asserted invariance. The cancellation of the second
displacement is the stationarity cancellation in
\eqref{eq:cubic-reduction}.

To compare normalizations along a fixed first minimizing curve, write
\[
 N_\tau=\tau^2n_2+\tau^3n_3+O(\tau^4),\qquad
 D_\tau=D_0+\tau d_1+O(\tau^2),\qquad D_0>0.
\]
Then
\([\tau^3]\frac{N_\tau}{D_\tau} =\frac{n_3}{D_0}-\frac{n_2d_1}{D_0^2}\).
Thus, when \(n_2=0\), two denominators with the same positive
zeroth-order value give the same cubic coefficient. The factor
\(1/D_0\) remains. The assumption \(q=0\) is precisely the vanishing
of the second numerator coefficient along the minimizing curve.
\end{proof}

Let \(D_h\) denote the first variation of the Levi--Civita connection.
Differentiating the Koszul formula, as in
\cite[Appendix~A]{BrendleHung2026}, gives
\begin{equation}\label{eq:connection-variation}
 2g_0(D_h(X,Y),W)
 =(\nabla_Xh)(Y,W)+(\nabla_Yh)(X,W)-(\nabla_Wh)(X,Y).
\end{equation}
At a point of \(\cZ_{\mathrm I}^{\mathrm{reg}}\), after diagonal
\(SO(4)\)-normalization,
choose \(X=E_i\), \(Y=F_i\) and the
transverse frame \(N=(E_j,E_k,F_j,F_k)\). Let \(A\) be the Hessian of
the unnormalized background curvature in the eight coefficients of
\(X+a\cdot N,Y+b\cdot N\). Let \(r_h\) be the gradient in those
coefficients of the first metric variation of this same
\emph{unnormalized curvature numerator}, evaluated at \(a=b=0\).
For \(h_1,h_2\in\ker\cL\),
\begin{equation}\label{eq:quadratic-connection-formula}
\begin{aligned}
 D_{\mathrm I}\cQ(h_1,h_2)={}&
 g_0(D_{h_1}(X,Y),D_{h_2}(X,Y)) -\tfrac12g_0(D_{h_1}(X,X),D_{h_2}(Y,Y))\\
 &-\tfrac12g_0(D_{h_2}(X,X),D_{h_1}(Y,Y))
 -\tfrac12r_{h_1}^TA^{-1}r_{h_2}.
\end{aligned}
\end{equation}
This is the polarized metric second-variation formula
\cite[Appendix~A]{BrendleHung2026}, followed by
\eqref{eq:schur-reduction} and division by \(D_{\mathrm I}\). In particular it includes the loss from
moving the plane. At degenerate fixed-base points the definition
\eqref{eq:metric-reduction}, with a full base-and-plane normal
complement, replaces this particular coordinate formula.

\subsection{First variations and separation of the two zero families}
\begin{lemma}\label{lem:first-null-tensors}
The smooth tensors \(H_0,O,B,C\) satisfy
\(\cL H_0=\cL O=\cL B=\cL C=0\) on \(\cZ\).
The two height summands \(O_p,O_q\) in \eqref{eq:odd-tensor}
separately satisfy \(\cL O_p=\cL O_q=0\).
\end{lemma}

\begin{proof}
We use the flat-torus restriction argument of
\cite[Proposition~5.1]{BrendleHung2026}, checking both zero families here.
Smoothness is immediate from the global formulas. On a type-I torus the
left great-circle axes and right great-circle axes are each constant.
The same is true on a type-II torus, with opposite right axes on its two
factors. Thus the pullbacks of \(B\) and \(C\) have constant tangential
components in the round-unit-speed circle coordinates, and
\eqref{eq:torus-linearization} vanishes.
For \(O\), its tangent entries in the frame \(E_u,F_u\) are
\(O_{ss}=p_0(s),\quad O_{tt}=q_1(t),\quad O_{st}=-\tfrac27\bigl(p_0(s)+q_1(t)\bigr)\).

Each derivative in \eqref{eq:torus-linearization} is again zero, also
for either summand separately.
On a type-I torus, \(v\) is parallel to \(u\), so \(P(v)u=0\);
on a type-II torus, \(r_0=0\). Hence its restriction vanishes on both
families, as required.
\end{proof}

\begin{proposition}\label{prop:H0-opening}
On \(\cZ_{\mathrm I}\), we have
\begin{equation}\label{eq:H0-radical}
 \cQ(H_0,h)=0\qquad\text{whenever }\cL h=0.
\end{equation}

On the second family, we have
\begin{equation}\label{eq:H0-minus-positive}
 \cQ(H_0)=\frac{282877}{2278125}>\frac1{10}.
\end{equation}

\end{proposition}

\begin{proof}
The tensor \(H_0\) is diagonal \(SO(4)\)-invariant. Every type-I torus
therefore remains totally geodesic for \(g_0+\tau H_0\), with unchanged
flat induced metric. For the commuting round-unit-speed coordinate fields
\(X,Y\) on such a torus, the ambient derivatives
\(\nabla^{g_0+\tau H_0}_XX\), \(\nabla^{g_0+\tau H_0}_XY\), and
\(\nabla^{g_0+\tau H_0}_YY\) vanish for every small \(\tau\). Hence
\(D_{H_0}(X,X)=D_{H_0}(X,Y)=D_{H_0}(Y,Y)=0 \quad\text{as ambient vectors}\).
Flatness and total geodesicity also give zero curvature differential
in the plane directions. At points of \(\cZ_{\mathrm I}^{\mathrm{reg}}\),
these directions together with \(T\cZ_{\mathrm I}\) span the ambient
tangent space. Since \(F_{1,H_0}|_{\cZ}=0\), it follows that
\(\dd F_{1,H_0}=0\) there, and continuity extends this to all of
\(\cZ_{\mathrm I}\). Formula
\eqref{eq:quadratic-connection-formula} proves \eqref{eq:H0-radical}.

The diagonal \(SO(4)\)-action is transitive on the second zero family
and preserves \(H_0\), so it suffices to use \(p=1,q=j,u=i\).
At this point \(H_0=0\) and
\[
 \dd H_0=(\alpha_2-\beta_2)\otimes
       \begin{pmatrix}0&P(j)\\P(j)&0\end{pmatrix}.
\]

Consequently \(D_{H_0}(E_i,E_i)=D_{H_0}(F_i,F_i)=0\), while
\(D_{H_0}(E_i,F_i)=-\frac12\operatorname{grad}_{g_0}r_0\).
Since \(|\dd r_0|_{g_0}^2=2/3\), the fixed quadratic numerator is
\(1/6\). The full normal Schur subtraction is
\(5107/151875\), as computed in \eqref{eq:typeII-Schur}; the
calculation includes the base direction changing \(r_0\). Dividing by
\(D_{\mathrm{II}}\) gives
\[
 \cQ(H_0)=\frac{1/6-5107/151875}{15/14}
          =\frac{282877}{2278125}.
\]

Transitivity can also be seen from the adapted orthonormal frames
\((p,pu,q,qu)\): all have the orientation of \((1,i,j,-k)\), and hence
lie in a single \(SO(4)\)-orbit.
\end{proof}

%% file: quadratic.tex
\section{The second-order curvature coefficient}\label{sec:quadratic}

We choose the first two perturbation tensors so that the second-order
coefficient is nonnegative on \(\cZ\), with zero set the tangent planes
to one flat torus. The principal estimate is the following calculation
for the height tensor \(O\); the mixed corrections will be chosen to
preserve it.

\begin{proposition}[Quadratic lower bound for the odd height tensor]\label{prop:odd-margin}
On a type-I torus, choose one oriented round-unit-speed parametrization
\(c:\R/2\pi\Z\to C_L\), write \(p=c(s)\), \(q=c(t)\), and set
\(\theta=t-s\), \(x=\sin^2\theta\).
Define the scalar polynomials
\begin{equation}\label{eq:height-polynomials}
\begin{aligned}
 d(x)&=31752000(25+3x),\\
 P(x)&=12632468+1767741x+5670x^2,\\
 V(x)&=3941932+686019x,\\
 W(x)&=3521524+636555x.
\end{aligned}
\end{equation}

Let \(J\) be the positive quarter-turn on \(L^\perp\), using the
orientation induced by the chosen orientation of \(L\) and the ambient
orientation of \(\R^4\). Then we have
\begin{equation}\label{eq:odd-identity}
\begin{aligned}
 \cQ(O)+\cL K_0={}&
 \frac{P(x)}{2d(x)}\rho
 +\frac{V(x)\cos\theta}{d(x)}\langle e_\perp,f_\perp\rangle\\
 &-\frac{W(x)\sin\theta}{d(x)}
                       \langle Je_\perp,f_\perp\rangle.
\end{aligned}
\end{equation}

In particular, we have
\begin{equation}\label{eq:odd-lower-bound}
 \cQ(O)+\cL K_0\geq\frac1{250}\rho
 \qquad\text{on }\cZ_{\mathrm I}.
\end{equation}

\end{proposition}

\begin{remark}
    The polynomials in \eqref{eq:height-polynomials} give the coefficients
    in \eqref{eq:odd-coefficient-matrices} after cancellation~\eqref{eq:K0-matching}.
\end{remark}

\begin{proof}
As in \cite[Proposition~5.2]{BrendleHung2026}, we expand the second
coefficient in the tensor amplitudes. Here the required blocks are
\eqref{eq:odd-coefficient-matrices} on the type-I family.
Use the diagonal \(SO(4)\)-action to put
\(p=1\), \(q=\cos\theta+i\sin\theta\). The two fixed height vectors
become arbitrary orthonormal vectors \(e,f\in\R^4\).
Formula \eqref{eq:quadratic-connection-formula} is quadratic in these
height vectors. The two self-coefficient matrices and the corrected
cross-coefficient matrix are given explicitly in
\eqref{eq:odd-coefficient-matrices}. They give
\eqref{eq:odd-identity}; the scalar multiple of
\(\langle e,f\rangle\) vanishes because \(e\perp f\). This calculation is
an exact identity over \(\Q(z)\), with
\(\cos\theta=(1-z^2)/(1+z^2)\) and
\(\sin\theta=2z/(1+z^2)\). Smoothness of the reduced coefficient extends
the identity to \(\theta=0,\pi\). Reversing the orientation of \(L\)
changes both \(J\) and \(\sin\theta\) in sign, so the formula is
independent of that choice.

For \(0\leq x\leq1\), the positive coefficients in
\eqref{eq:height-polynomials} give the endpoint bounds
\begin{equation}\label{eq:height-polynomial-bounds}
\begin{gathered}
 P(x)\geq12632468,\qquad
 \max\{V(x),W(x)\}\leq4627951,\\
 0<31752000\cdot25\leq d(x)\leq31752000\cdot28.
\end{gathered}
\end{equation}
The vectors \(e_\perp,Je_\perp\) are orthogonal and have equal length.
Consequently the absolute value of the last two terms in
\eqref{eq:odd-identity} is at most
\(\max\{V,W\}|e_\perp||f_\perp|/d\).
Since \(2|e_\perp||f_\perp|\leq\rho\), the coefficient of \(\rho\) is
bounded below by
\begin{equation}\label{eq:height-numerical-margin}
 \frac{12632468-4627951}{2\cdot31752000\cdot28}>\frac1{250}.
\end{equation}
This proves \eqref{eq:odd-lower-bound}.
\end{proof}

\input{mixed-correction}

\subsection{Preservation of the quadratic lower bound}

The corrections involving \(B,C\) must preserve
\eqref{eq:odd-lower-bound}. Specifically, each function in
\eqref{eq:five-corrections} must vanish together with its first
derivative on \(\Tlift\); for \(K_1\) this is
\eqref{eq:admissible-mixed}, and for \(K_2,K_3\) it follows from
\eqref{eq:BB-cancel} and~\eqref{eq:BC-j}--\eqref{eq:BC-k}.

\begin{proposition}[Vanishing transverse jets]\label{prop:mixed-jets}
Each of the five functions in \eqref{eq:five-corrections} vanishes
on \(\Tlift\), together with its first derivative along
\(\cZ_{\mathrm I}\).
\end{proposition}

\begin{proof}
The reflection \eqref{eq:reflection} fixes \(\Tlift\) pointwise and
reverses the four transverse directions in
\(\cZ_{\mathrm I}\). Under this reflection \(O,C,K_0,K_2\) are even,
whereas \(B,K_1,K_3\) are odd.
The exact identities on \(\Tlift\) are
\begin{equation}\label{eq:BB-cancel}
 \cQ(B)=\frac{128}{845}\cos2\theta,\qquad
 \cL K_2=-\frac{128}{845}\cos2\theta,
 \qquad \cQ(C)=\cQ(O,C)=0.
\end{equation}

Thus the even functions in \eqref{eq:five-corrections} have both zero
value and zero first transverse derivative.

For the \(O\)-\(B\) assertion, use condition~\eqref{eq:admissible-mixed}.
For the remaining odd function,
parity gives zero value but not zero derivative.
Set
\(D_j=2\alpha_1\odot\alpha_2,\quad D_k=2\alpha_1\odot\alpha_3,\quad S(U,V)=Ug_0^{-1}V+Vg_0^{-1}U\).
Let \(A_{\mu\nu}\) send \(e_\mu\) to \(e_\nu\), \(e_\nu\) to
\(-e_\mu\), and annihilate the other basis vectors. For any of
\(A_{02},A_{03},A_{12},A_{13}\), let \(\Phi_a\) be its diagonal
\(SO(4)\)-flow and define
\(\dot B=\left.\frac{\dd}{\dd a}\right|_{0}\Phi_a^*B, \quad \dot C=\left.\frac{\dd}{\dd a}\right|_{0}\Phi_a^*C\).
Since \(\Phi_a\) preserves \(g_0\) and \(\cZ\), the operators
\(\cL,\cQ\) are equivariant under this pullback, and the differentiated
tensors belong to \(\ker\cL\).
The derivative along the induced
flow on \(\cZ_{\mathrm I}\) is therefore
\[
\begin{aligned}
 \left.\frac{\dd}{\dd a}\right|_0
 \Phi_a^*\bigl(2\cQ(B,C)+\cL K_3\bigr)
 ={}&2\cQ(\dot B,C)+2\cQ(B,\dot C)
 +\frac85\cL S(\dot B,C)+\frac85\cL S(B,\dot C).
\end{aligned}
\]

For a quaternionic flow \(p\mapsto e^{a\ell}p e^{ab}\), the left
coframe satisfies \(\dot\alpha=-[b,\alpha]\). The four rotations
above give
\[
\begin{array}{c|cccc}
 &A_{02}&A_{03}&A_{12}&A_{13}\\ \hline
 b&j/2&k/2&-k/2&j/2\\
 \dot C&-D_k&D_j&-D_j&-D_k
\end{array}
\]
because \(\dot C=2\alpha_1\odot\dot\alpha_1\).

We first dispose of the terms containing \(\dot B\) and \(C\).
The reflection \eqref{eq:reflection} preserves \(g_0+\tau C\), so
\(T_0\) remains totally geodesic for small \(\tau\). Its induced
metric has constant coefficients \(G_{\mathrm I}+\tau\operatorname{diag}(1,0)\)
in \(s,t\). Consequently, for \(X=E_i\), \(Y=F_i\) along \(T_0\),
\(D_C(X,X)=D_C(X,Y)=D_C(Y,Y)=0 \quad\text{as full ambient vectors}\).
Flatness and total geodesicity also give zero differential in the plane
directions for every metric in this path. Since \(F_{1,C}|_{\cZ}=0\),
these directions together with \(T\cZ\) imply
\(\dd F_{1,C}=0\) on \(\Tlift\cap\cZ_{\mathrm I}^{\mathrm{reg}}\).
Continuity extends this to all of \(\Tlift\). Formula
\eqref{eq:quadratic-connection-formula} now gives
\(\cQ(\dot B,C)=0\) there.

The tangential components of \(\Phi_a^*B\) on \(T_0\) are constant
in \(s,t\), since \(\Phi_a\) maps it to a type-I flat torus. Hence
the same is true of \(\dot B\). Along \(T_0\), the \(g_0\)-normal
space is spanned by \(E_j,E_k,F_j,F_k\), and \(C\) annihilates this
space. It follows that the tangential restriction of \(S(\dot B,C)\)
has constant coefficients as well. Thus
\(\cL S(\dot B,C)=0\) by \eqref{eq:torus-linearization}.

For the remaining terms, exact evaluation gives
\begin{align}
 2\cQ(B,D_j)&=\frac{128}{325}\cos2\theta,&
 \cL S(B,D_j)&=-\frac{16}{65}\cos2\theta,\label{eq:BC-j}\\
 2\cQ(B,D_k)&=\frac{128}{325}\sin2\theta,&
 \cL S(B,D_k)&=-\frac{16}{65}\sin2\theta.
 \label{eq:BC-k}
\end{align}
Thus the coefficient \(8/5\) chosen by \eqref{eq:K3-matching}
cancels these terms for each of the four values of \(\dot C\),
since \(128/325-(8/5)(16/65)=0\).

The identities \eqref{eq:BC-j}--\eqref{eq:BC-k} result from
substituting \(B,D_j,D_k\) into
\eqref{eq:quadratic-connection-formula} and~\eqref{eq:torus-linearization},
using the frame calculus in Appendix~\ref{app:frame-calculus}. The four rotations span
\(\operatorname{Hom}(L_0,L_0^\perp)\). Since \(\Tlift\) is the fiber
over \(L_0\) of the supporting-plane map
\(\cZ_{\mathrm I}\to\Gr_2(\R^4)\), they cover its four transverse
directions, including above \(p=\pm q\). The function vanishes on
\(\Tlift\), so its tangential derivatives vanish too.
\end{proof}

For the functions \(F,G\) in \eqref{eq:FG}, we now justify the
finiteness of \(C_*\) and the amplitude choices in
Section~\ref{sec:amplitude-choice}.
Write \(L_U=\operatorname{graph}U\) in the graph coordinates of
Section~\ref{sec:K1-choice}. Orthogonal projection gives
\(\rho(L_U)=\operatorname{tr}\bigl[U^TU(I_{L_0}+U^TU)^{-1}\bigr] =\|U\|_{\mathrm{HS}}^2+O(\|U\|_{\mathrm{HS}}^4)\).
Proposition~\ref{prop:mixed-jets} and Taylor's theorem give
\(|F|+|G|\leq C\|U\|_{\mathrm{HS}}^2\) near \(\Tlift\), uniformly
in its compact set of angular variables. Hence \(|F|+|G|\leq C'\rho\)
there. On the compact complement of this neighborhood, \(\rho\) has
a positive minimum. Thus \(F,G=O(\rho)\) uniformly on \(\cZ_{\mathrm I}\),
so \(C_*\) in \eqref{eq:Cstar} is finite.

Use \(\eps,h_*,K\) from \eqref{eq:small-tensor-choice}.
Then \(\eps\leq1\), and we have
\begin{equation}\label{eq:plus-margin}
 \cQ(h_*)+\cL K
 \geq\frac{\rho}{250}-\eps(|F|+|G|)
 \geq\frac3{1000}\rho
 \qquad\text{on }\cZ_{\mathrm I}.
\end{equation}
The remaining constants \(M_*,c_*\) and the tensor \(H\) are
those of \eqref{eq:Mstar}--\eqref{eq:H-final}.

\begin{proposition}\label{prop:quadratic-final}
For \(H,K\) in \eqref{eq:small-tensor-choice}--\eqref{eq:H-final}, we have
\begin{equation}\label{eq:q2}
 \Vtwo:=\cQ(H)+\cL K\geq0,\qquad \{\Vtwo=0\}=\Tlift.
\end{equation}
Its first derivative along \(\cZ\) vanishes on \(\Tlift\).
\end{proposition}

\begin{proof}
By \eqref{eq:H0-radical}, adding \(c_*H_0\) does not change the
quadratic coefficient on \(\cZ_{\mathrm I}\). There
\eqref{eq:plus-margin} is positive off \(\Tlift\), and
Propositions~\ref{prop:odd-margin} and~\ref{prop:mixed-jets} give zero
on \(\Tlift\).
On the second component, \eqref{eq:H0-minus-positive}
gives the lower bound
\begin{equation}\label{eq:typeII-amplitude-bound}
 \cQ(H)+\cL K
 \geq\frac{c_*^2}{10}-c_*M_*-M_*
 =20M_*^2-M_*>0
 \quad\text{on }\cZ_{\mathrm{II}}.
\end{equation}

Finally, a smooth nonnegative function has zero differential at each
point of its zero set.
\end{proof}

%% file: mixed-correction.tex
\subsection{A global mixed correction}\label{sec:mixed-correction}

The interaction of \(O\) with \(B\) need not respect the quadratic margin
near \(\Tlift\). We construct a correction that cancels its first transverse
jet. A tangent--normal tensor component is essential: at \(p=\pm q\),
the motion of the base points does not determine that of the supporting
two-plane.

We verify the tensor specified in Section~\ref{sec:K1-choice}.

\begin{proposition}\label{prop:mixed-existence}
The tensor \(K_1\) in \eqref{eq:geometric-K1} is smooth, has the
parities specified in Section~\ref{sec:K1-choice}, and satisfies
\eqref{eq:admissible-mixed}, including at every pair \(p=\pm q\).
\end{proposition}

\begin{proof}
The mixed cancellations parallel those in
\cite[Proposition~5.2]{BrendleHung2026}; here we realize the full
transverse derivative by a global tensor, including at \(p=\pm q\).
Use the graph coordinates and the notation
\(p_s,q_t,N,U,A_U\) of Section~\ref{sec:K1-choice}.
For the two height summands, set
\(\mathscr S_p=2\cQ(O_p,B)\) and
\(\mathscr S_q=2\cQ(O_q,B)\).
These are smooth functions on the full zero family, by
Section~\ref{sec:reduction}. Their reflection oddness gives zero on
\(\Tlift\). Write their first graph derivatives as
\(\mathscr S_p(s,t,U)=L_p(s,t)[U]+O(|U|^2),\quad \mathscr S_q(s,t,U)=L_q(s,t)[U]+O(|U|^2)\).
The remainder estimates hold with angular derivatives. The functional
\(L_p\) is \(\pi\)-antiperiodic in \(s\) and \(\pi\)-periodic in \(t\);
the parities of \(L_q\) are reversed. Indeed, \(O_p\) is odd only under
the first-factor antipodal map, \(O_q\) only under the second, and \(B\)
is even under both.

For a smooth \(\pi\)-antiperiodic function \(f:\R\to\R\), define
\(\mathcal I f(s)=\int_0^\pi(\xi-\pi/2)f(s+\xi)\,\dd\xi\).
Integration by parts and antiperiodicity give
\[
 (\mathcal I f)'(s)=-\int_0^\pi f(s+\xi)\,\dd\xi,\qquad
 (\mathcal I f)''(s)=2f(s).
\]
It follows that the functionals \(\kappa_p,\kappa_q\) defined in
\eqref{eq:mixed-primitives}
satisfy
\begin{equation}\label{eq:mixed-primitive-equations}
 \partial_s^2\kappa_p=2D_{\mathrm I}L_p=\tfrac{15}{4}L_p,
 \qquad
 \partial_t^2\kappa_q=\tfrac{15}{4}L_q.
\end{equation}

All parities are preserved. Smooth parameter dependence follows by
differentiation under a fixed compact-interval integral.
The Hilbert--Schmidt representatives \(D_p,D_q\) are therefore
smooth as well. In the tubular coordinates of
Section~\ref{sec:K1-choice}, take \(k_p,k_q,K_1\) from
\eqref{eq:geometric-K1}.
Periodicity makes these genuine tensors on the neighborhood, and the
cutoff gives a smooth extension by zero.

On the graph torus, with \(X=\partial_s\) and \(Y=\partial_t\),
\(y=Uq_t+O(|U|^2),\quad \dd y(Y)=Uq_t'+O(|U|^2),\quad \dd t(Y)=1+O(|U|^2)\).
Since \(k_p\) has only second-factor covector slots,
\(k_p(X,X)=k_p(X,Y)=0\), whereas
\begin{equation}\label{eq:mixed-jet-recovery}
\begin{aligned}
 k_p(Y,Y)
 &=\langle D_pq_t,Uq_t\rangle
   +\langle D_pq_t',Uq_t'\rangle+O(|U|^2)\\
 &=\langle D_p,U\rangle_{\mathrm{HS}}+O(|U|^2)
 =\kappa_p[U]+O(|U|^2).
\end{aligned}
\end{equation}
The second equality uses the orthonormal basis \(\{q_t,q_t'\}\) of
\(L_0\). The estimates again hold after angular differentiation.
Equations \eqref{eq:torus-linearization} and
\eqref{eq:mixed-primitive-equations} imply
\(\mathscr S_p+\cL k_p=O(|U|^2),\quad \mathscr S_q+\cL k_q=O(|U|^2)\).
This proves the value and differential conditions in
\eqref{eq:admissible-mixed}, including derivatives tangent to \(\Tlift\).

Finally, \(k_p\) is odd under the first antipodal map and even under
the second; \(k_q\) has the opposite parities. Both are odd under
\eqref{eq:reflection}. These statements follow directly from
\eqref{eq:mixed-primitives}--\eqref{eq:geometric-K1}; the radial cutoff
preserves them. This proves all the required parities.
\end{proof}

At \(p=\pm q\), a graph direction may satisfy \(Uq_t=0\) while
\(Uq_t'\ne0\). The tangent--normal term in
\eqref{eq:mixed-jet-recovery} records precisely this remaining direction.
No division by \(\sin(t-s)\) occurs in the construction.

The argument applies to any fixed \(K_1\) satisfying
\eqref{eq:admissible-mixed} and the stated parities. Appendix~\ref{app:explicit-K1}
gives a second, algebraic choice that is useful for direct computation.
Both this tensor and the geometric choice \eqref{eq:geometric-K1}
satisfy the conditions used to prove the quadratic margin and cubic mean.

%% file: cubic.tex
\section{The third-order curvature coefficient}\label{sec:cubic}

The second coefficient \(\Vtwo\) now vanishes only on \(\Tlift\). For the
chosen tensors \(H,K\), write \(\Vthree=\cR(H,K)\). Following
Brendle and Hung~\cite[Proposition~5.5]{BrendleHung2026}, we determine
the average of \(\iota_0^*\Vthree\) on \(T_0\) by expanding in
the tensor amplitudes. The average depends only on the following coefficient
in the amplitude expansion of \(\Vthree\).

\begin{proposition}\label{prop:basic-cubic}

On \(\Tlift\), with \(p=e^{si}\), \(q=e^{ti}\), and \(\theta=t-s\), we have
\begin{equation}\label{eq:basic-cubic}
 \cR(B+C,K_2+K_3)
 =\frac{512}{375}+\frac{9728}{21125}\cos2\theta.
\end{equation}

The corrected quadratic coefficient and its first derivative along
\(\cZ_{\mathrm I}\) vanish on \(\Tlift\).
\end{proposition}

\begin{proof}
The quadratic statement follows from
Proposition~\ref{prop:mixed-jets}. Common left multiplication by
\(e^{ai}\) preserves \(g_0,B,C,K_2,K_3\), so it suffices to compute
at \(p=1,q=e^{\theta i}\); the answer applies to every common phase.

At a point of \(\cZ_{\mathrm I}^{\mathrm{reg}}\), solving
\(A_z\eta_1+b_z=0\) as in \eqref{eq:schur-reduction} gives the first
minimizing graph displacement. In the order
\((a_{E_j},a_{E_k},a_{F_j},a_{F_k},b_{E_j},b_{E_k},b_{F_j},b_{F_k})\), it is
\begin{equation}\label{eq:singular-displacement}
 \begin{aligned}
  \eta_1&=(a,b),\\
 65a&=(-96,\ 32\cot\theta,\ -64+24\sin^2\theta,\
                            24\sin\theta\cos\theta),\\
 65b&=(64-24\sin^2\theta,\ 24\sin\theta\cos\theta,\
                            96,\ 32\cot\theta).
 \end{aligned}
\end{equation}
The eight stationarity equations and the cubic substitution in
\eqref{eq:cubic-reduction} give the exact coefficient
\begin{equation}\label{eq:cubic-rational}
 \frac{115712z^4-2048z^2+115712}{63375(1+z^2)^2},
 \qquad
 \cos\theta=\frac{1-z^2}{1+z^2},\quad
 \sin\theta=\frac{2z}{1+z^2},
\end{equation}
which is \eqref{eq:basic-cubic}. The calculation is described in
Appendix~\ref{app:cubic-calculation}.

The apparent poles in \eqref{eq:singular-displacement} come from the
choice of fixed-base complement. Lemma~\ref{lem:change-section} and the
smooth displacement \eqref{eq:smooth-base-displacement}--\eqref{eq:smooth-V}
remove them.
Equivalently, the full normal
calculation at \(p=q=1\) gives zero quadratic coefficient and cubic
coefficient \(115712/63375\), agreeing with
\eqref{eq:basic-cubic}.
 Smoothness gives the identity everywhere on
\(\Tlift\).
\end{proof}

We describe the displacement without dividing by \(\sin\theta\). Put
\begin{align}
 w_\theta&=-\frac{32}{65}
     \bigl(\sin\theta\cos\theta\,j+\cos^2\theta\,k\bigr),
       \label{eq:smooth-base-displacement}\\
 U_\theta&=\frac1{65}
   \bigl(-96E_j-(64-24\sin^2\theta)F_j
                    +24\sin\theta\cos\theta\,F_k\bigr),\notag\\
 V_\theta&=\frac1{65}
   \bigl((64-24\sin^2\theta)E_j
                    +24\sin\theta\cos\theta\,E_k+96F_j\bigr).
       \label{eq:smooth-V}
\end{align}

For the actual first tensor \(H\), define
\begin{equation}\label{eq:smooth-moving-plane}
 p_\tau=p,\qquad q_\tau=q\exp(\tau\eps w_\theta),\qquad
 \sigma_\tau=\Span\{E_i+\tau\eps U_\theta,\,
                    F_i+\tau\eps V_\theta\}
       \subset T_{(p_\tau,q_\tau)}M.
\end{equation}
All frame fields in the last expression are evaluated at the moved point.

\begin{lemma}\label{lem:pole-free}
The curve \eqref{eq:smooth-moving-plane} has a minimizing first normal
displacement along \(\Tlift\). Its cubic curvature coefficient for
\(g_0+\tau H+\tau^2K\) is a smooth periodic function of \(s,t\), equal
to the reduced cubic coefficient.
\end{lemma}

\begin{proof}
The invariance needed here is
\cite[Proposition~2.9]{BrendleHung2026}; we give an explicit smooth
representative of the first displacement.
Hold \(p=e^{si}\) fixed and vary the type-I torus axis by \(\dot u=k\).
For \(q=p(\cos\theta+\sin\theta\,u)\), this tangent variation to
\(\cZ_{\mathrm I}\) has
\(q^{-1}\dot q=\sin^2\theta\,j+\sin\theta\cos\theta\,k, \quad (\dot X,\dot Y)=(E_k,F_k)\).

Subtracting \(32\cot\theta/65\) times this tangent vector from
\eqref{eq:singular-displacement} gives exactly
\eqref{eq:smooth-base-displacement}--\eqref{eq:smooth-V}.
The result is smooth at \(\theta=0,\pi\), and the stationarity equations
extend there by continuity in the full normal bundle.

Only \(B\) contributes to the first normal displacement on \(\Tlift\).
For \(H_0\) this follows from \eqref{eq:H0-radical} and its proof;
for \(C\) it follows from the ambient connection and differential
identities proved in Proposition~\ref{prop:mixed-jets}. For \(O\), the reflection
\eqref{eq:reflection} keeps \(T_0\) totally geodesic, so Codazzi gives
zero first plane-curvature gradient; regular normal complements and
continuity give \(\dd F_{1,O}=0\) on \(\Tlift\). Thus the displacement
for \(H\) is \(\eps\) times that for \(B+C\).
Since \(\Vtwo=\dd \Vtwo=0\) on \(\Tlift\), Lemma~\ref{lem:change-section}
identifies the cubic coefficient with the reduced one. The explicit
formula makes smoothness and periodicity immediate.
\end{proof}

Denote the function in Lemma~\ref{lem:pole-free} by
\(\Vthreebase=\iota_0^*\Vthree,\quad \Vthreebase(s,t)=\Vthree\bigl(\iota_0(e^{si},e^{ti})\bigr)\).
We take averages with respect to the constant background area form on
\(T_0\), normalized to total mass one.

\begin{proposition}\label{prop:cubic-mean}
For every \(K_1\) satisfying \eqref{eq:admissible-mixed}
and the parities in Section~\ref{sec:K1-choice}, the choices
\eqref{eq:small-tensor-choice}--\eqref{eq:H-final} give
\begin{equation}\label{eq:cubic-mean}
 \mean_{T_0}\Vthreebase=\frac{512}{375}\eps^3=:m>0.
\end{equation}

\end{proposition}

\begin{proof}
We follow the amplitude-polynomial argument of
\cite[Proposition~5.5]{BrendleHung2026}; the parities here isolate
exactly the \(\delta^2\mu\) term.
Introduce real amplitudes \(\lambda,\delta,\mu,\gamma\) and set
\begin{align*}
 h&=\lambda O+\delta B+\mu C+\gamma H_0,\\
 k&=\lambda^2K_0+\lambda\delta K_1+\delta^2K_2+\delta\mu K_3.
\end{align*}
The corrected quadratic coefficient and its first derivative vanish
on \(\Tlift\) for every amplitude choice. This follows term by term from
Propositions~\ref{prop:H0-opening}, \ref{prop:odd-margin},
and~\ref{prop:mixed-jets}. The pullback of the reduced cubic coefficient to \(T_0\) is
therefore an intrinsically defined homogeneous cubic polynomial in
these amplitudes.
The reflection \eqref{eq:reflection} changes \(\delta\) to
\(-\delta\), so this polynomial is even in \(\delta\).

When \(\delta=0\), the reflection preserves the entire metric family and
fixes \(T_0\), which remains totally geodesic. Its intrinsic curvature
coefficients of orders zero, one, and two vanish. The reduced cubic
equals its intrinsic cubic, since the first normal displacement is
zero. More explicitly, write the intrinsic curvature and area form as
\[
 K_\tau^{\mathrm{int}}=\tau^3r_{\mathrm{int}}+O(\tau^4),\qquad
 \dd A_\tau=(1+\tau a_1+O(\tau^2))\,\dd A_0.
\]
The Gauss--Bonnet theorem gives
\[
 0=\int_{T_0}K_\tau^{\mathrm{int}}\,\dd A_\tau
   =\tau^3\int_{T_0}r_{\mathrm{int}}\,\dd A_0+O(\tau^4).
\]
Thus the cubic coefficient has zero average with respect to the
background area form. Every monomial not involving \(\delta\) therefore
has zero average.

The remaining possibilities are
\(\delta^2\lambda,\delta^2\mu,\delta^2\gamma\). The simultaneous
antipodal map reverses \(O,K_1\) and preserves the other relevant
tensors and the torus area, so the coefficient of
\(\delta^2\lambda\) has zero mean. To isolate the
\(\delta^2\gamma\) coefficient, set \(\lambda=0\).
The first-factor antipodal map then reverses \(H_0\) and preserves
\(B,C,K_2,K_3\), so this coefficient also has zero mean.

It remains to determine the coefficient of \(\delta^2\mu\).
Set \(\lambda=\gamma=0\), \(\delta=\mu=1\). Proposition~\ref{prop:basic-cubic}
gives mean \(512/375\), because
\(\cos2(t-s)\) has zero mean. For the actual tensors,
\(\lambda=1\), \(\delta=\mu=\eps\), and \(\gamma=c_*\).
This proves \eqref{eq:cubic-mean}.
\end{proof}

Only the prescribed derivatives and parity of \(K_1\) entered this
argument. At fixed tensor amplitudes, replacing \(K_1\) may change the
pointwise function \(\Vthreebase\), but not its average. If the auxiliary
constants are reselected for a new \(K_1\), the value of \(\eps\) may also
change; the identity \eqref{eq:cubic-mean} still holds with that value.
The Poisson equation uses the corresponding function \(\Vthreebase\).

%% file: positivity.tex
\section{Completion of the construction}\label{sec:positivity}

We justify the conformal correction in Section~\ref{sec:complete-metric}
and then prove Theorem~\ref{thm:construction}. This is the final step of
the method in \cite[Section~6]{BrendleHung2026}.
Proposition~\ref{prop:cubic-mean} gives the zero-mean right-hand side
required in \eqref{eq:poisson}. With the sign convention in
\eqref{eq:torus-laplacian}, the solution is as follows.
For completeness, if \(\widehat{V}_{0,ab}^{(3)}\) are the Fourier
coefficients of \(\Vthreebase\) in the \(2\pi\)-periodic coordinates
\(s,t\), then
\[
 f(s,t)=-60\sum_{(a,b)\in\Z^2\setminus\{(0,0)\}}
  \frac{\widehat{V}_{0,ab}^{(3)}e^{i(as+bt)}}
       {29(a^2+b^2)+38ab}.
\]
The coefficient \(-60\) is \(2\) from \eqref{eq:poisson}
divided by the common factor \(-1/30\) in the Fourier symbol of
\(\Delta_T\), as given by \eqref{eq:torus-laplacian}.
The denominator equals \(10(a^2+b^2)+19(a+b)^2\) and is positive
away from the zero character.
 Smoothness of \(\Vthreebase\) implies
convergence with all derivatives. The usual conjugacy of the Fourier
coefficients makes \(f\) real-valued.

The cutoff in \eqref{eq:smooth-extension} vanishes on a neighborhood
of every point where a projection is undefined. Thus the extension is
smooth on \(M\), and it equals \(f\) on \(T_0\).

\begin{proof}[Proof of Theorem~\ref{thm:construction}]
We apply the conformal-correction argument of
\cite[Proposition~6.1]{BrendleHung2026}; Lemma~\ref{lem:normal-minimum}
provides its uniform completion on both zero components.
All tensors and constants have been fixed before choosing \(\tau\).

Since \(g_0\geq(5/8)I_6\) and the perturbation tensors are bounded on
the compact manifold, \(g_\tau\) is positive definite for sufficiently
small \(|\tau|\).

On \(T_0\), which is flat and totally geodesic for the background metric,
the conformal first-variation formula is
\begin{equation}\label{eq:conformal-linearization}
 \iota_0^*\cL(\widetilde f\,g_0)=-\tfrac12\Delta_Tf.
\end{equation}

The term of order \(\tau^3\) in \eqref{eq:final-metric} therefore
changes the pullback of the reduced cubic to
\(\iota_0^*\Vthreecorr=\Vthreebase-\tfrac12\Delta_Tf=m>0\),
where we write
\(\Vthreecorr=\Vthree+\cL(\widetilde f\,g_0)\) on \(\cZ\).
It does not change the second coefficient \(\Vtwo\).
The full normal minimum has a uniform expansion
\begin{equation}\label{eq:final-normal-minimum}
 \tau^2\Vtwo(z)+\tau^3\Vthreecorr(z)+O(\tau^4),\qquad z\in\cZ,
\end{equation}
where \(\Vtwo\geq0\), \(\{\Vtwo=0\}=\Tlift\), and \(\Vthreecorr=m>0\) on
\(\Tlift\). These are precisely the hypotheses of
Lemma~\ref{lem:normal-minimum}, using the positive full normal Hessian
in Proposition~\ref{prop:clean-zero}. It follows that every plane in
a fixed tube about \(\cZ\) has positive curvature for all sufficiently
small positive \(\tau\), uniformly in its base point. The compact
complement has a positive background curvature minimum, which also
persists. Decreasing \(\tau_0\) accommodates both regions and proves
the asserted strict inequality on the whole Grassmann bundle.
\end{proof}

Theorem~\ref{thm:main} follows. Since each three-sphere has Euler
characteristic zero and Euler characteristic is multiplicative for
products of finite CW complexes, this also gives
\(\chi(S^3\times S^3)=0\) for a closed even-dimensional positively
curved manifold.

%% file: calculations.tex
\section{Exact curvature identities and their verification}
\label{app:calculations}

This appendix specifies the finite calculations used in
Propositions~\ref{prop:H0-opening}, \ref{prop:odd-margin},
\ref{prop:mixed-jets}, and~\ref{prop:basic-cubic}.
The accompanying code evaluates these identities over \(\Q(z)\), not
by numerical sampling. The geometric arguments establishing the zero
locus, the appropriate normal spaces, and uniform positivity are separate
from these algebraic checks.

\subsection{Frame calculus}\label{app:frame-calculus}

In this subsection only, order the tangent frame as
\((e_1,\ldots,e_6)=(E_i,E_j,E_k,F_i,F_j,F_k)\), and write
\([e_i,e_j]=c_{ij}^{m}e_m\), summing repeated indices. Write \(R^g\)
for the curvature endomorphism, to distinguish it from
\(R=\Ad_{p^{-1}q}\). We use \eqref{eq:curvature-convention}, so
\(\Sec_g(X\wedge Y)=g(R^g(X,Y)Y,X)/|X\wedge Y|_g^2\).
For any metric matrix \(g\), the noncoordinate formulas are
\begin{align}
 2g_{k\ell}\Gamma_{ij}^{\ell}
 ={}&e_i g_{jk}+e_j g_{ik}-e_k g_{ij}
       +c_{ij}^{m}g_{mk}-c_{jk}^{m}g_{mi}
       +c_{ki}^{m}g_{mj},\label{eq:koszul-engine}\\
 (R^g)_{ijk}{}^\ell
 ={}&e_i\Gamma_{jk}^{\ell}-e_j\Gamma_{ik}^{\ell}
       +\Gamma_{jk}^{m}\Gamma_{im}^{\ell}
       -\Gamma_{ik}^{m}\Gamma_{jm}^{\ell}
       -c_{ij}^{m}\Gamma_{mk}^{\ell}.
       \label{eq:curvature-engine}
\end{align}
All spatial two-jets follow from quaternion multiplication and
\(E_a r=-ar\), \(F_a r=ra\).
For example, if \(J_a b=a\times b\), then
\(E_aR=-2J_aR,\quad F_aR=2RJ_a,\quad E_bE_aR=4J_aJ_bR\).
The last order of multiplication is intentional. The code uses ordered
frame derivatives and checks their commutators against the structure
constants.

For a metric series, the coefficients of the lower connection in
\eqref{eq:koszul-engine} depend linearly on the metric coefficients.
The upper connection coefficients are then obtained recursively by
solving against the fixed zeroth-order metric. Differentiating this
recurrence and substituting into \eqref{eq:curvature-engine} gives the
curvature coefficients. This supplies one direct method of computing
both the quadratic and cubic identities.

A separate check contracts curvature before expanding the full tensor.
For frame fields with constant coefficients, let
\(B_g(X,Y,Z)=g(\nabla_XY,Z)\) and \(c=[X,Y]\). Metric compatibility
and torsion freeness give
\begin{equation}\label{eq:contracted-curvature}
\begin{aligned}
 g(R^g(X,Y)Y,X)={}&X\bigl(B_g(Y,Y,X)\bigr)
                   -Y\bigl(B_g(X,Y,X)\bigr)\\
 &+B_g(X,Y,\cdot)^Tg^{-1}B_g(X,Y,\cdot)\\
 &-B_g(X,X,\cdot)^Tg^{-1}B_g(Y,Y,\cdot)\\
 &-B_g(X,Y,c)-B_g(c,Y,X).
\end{aligned}
\end{equation}
At a background flat torus, the three lower connection vectors in the
Gram terms vanish at order zero. Their contribution through order three
therefore requires only the first derivative of \(g^{-1}\).
This separate check uses \eqref{eq:contracted-curvature} without invoking
the other cubic evaluators.

\subsection{The quadratic coefficient of the odd height tensor}\label{app:height-calculation}

At \(p=1\), \(q=\cos\theta+i\sin\theta\), write \(A_p,A_q\) for the
constant matrices multiplying \(p_0,q_1\) in \eqref{eq:odd-tensor},
and put \(X_0=\left(\begin{smallmatrix}0&I\\I&0\end{smallmatrix}\right)\).

For arbitrary \(e,f\in\R^4\), substitute
\begin{align*}
 h(e,f)&=\langle e,p\rangle A_p+\langle f,q\rangle A_q,\\
 k(e,f)&=-\tfrac15\langle e,p\rangle\langle f,q\rangle X_0
     -\tfrac{787}{5880}\langle f,p\rangle\langle e,q\rangle I_6.
\end{align*}
The result is a quadratic expression
\(\cQ(h(e,f))+\cL k(e,f) =e^T\mathsf M_p e+f^T\mathsf M_q f+e^T\mathsf M_{pq}f\).
With the polynomials in Proposition~\ref{prop:odd-margin}, exact
substitution gives
\begin{equation}\label{eq:odd-coefficient-matrices}
\begin{gathered}
 \mathsf M_p=\mathsf M_q
    =\operatorname{diag}\left(0,0,\frac{P}{2d},\frac{P}{2d}\right),\\
 \mathsf M_{pq}
 =d_0 I_4+
 \begin{pmatrix}
 0&0&0&0\\
 0&0&0&0\\
 0&0&V\cos\theta/d&-W\sin\theta/d\\
 0&0&W\sin\theta/d&V\cos\theta/d
 \end{pmatrix},
 \qquad d_0=-\frac{787}{11025}\cos\theta.
\end{gathered}
\end{equation}
The verification checks every entry of these three matrices as a
rational function. The orthogonality \(e\perp f\), supplied by the
diagonal \(SO(4)\)-normalization, eliminates the \(d_0 I_4\) term.
This proves the identity for every type-I torus.

\subsection{An algebraic mixed correction}\label{app:explicit-K1}

For the executable calculations we use the following algebraic choice
of \(K_1\). Like the integral construction in
Section~\ref{sec:mixed-correction}, it satisfies
\eqref{eq:admissible-mixed} and the parities in
Section~\ref{sec:K1-choice}.

Let \(\ell,\widehat\ell\) be either coframe pair
\((\alpha,\widehat\alpha)\) or \((\beta,\widehat\beta)\), and put
\begin{align*}
 U_\ell&=\ell_1+\widehat\ell_1,&
 V_\ell^\pm&=\ell_2\pm\widehat\ell_2,&
 W_\ell^\pm&=\tfrac{13}{7}\ell_3\pm\widehat\ell_3,\\
 T_\ell&=\ell_1^2+\ell_1\odot\widehat\ell_1+\widehat\ell_1^2.
\end{align*}
For a quaternion coordinate vector \(x=(x_0,x_1,x_2,x_3)\), define
\begin{align}
 \mathcal A_\ell(x)
 &=U_\ell\odot(x_0V_\ell^-+x_1W_\ell^-)-\tfrac47x_3T_\ell,
 \label{eq:A-local}\\
 \mathcal B_\ell(x)
 &=U_\ell\odot(x_0W_\ell^+-x_1V_\ell^+)-\tfrac47x_2T_\ell.
 \label{eq:B-local}
\end{align}
Define the algebraic tensor by
\begin{equation}\label{eq:K1}
\begin{aligned}
 K_1^{\mathrm{alg}}={}&\tfrac{11}{325}
       \bigl(\mathcal A_\beta(p)+\mathcal B_\alpha(q)\bigr)
       +\tfrac6{65}r_0
       \bigl(\mathcal B_\alpha(p)+\mathcal A_\beta(q)\bigr).
\end{aligned}
\end{equation}
In \(\mathcal A_\ell(x)\) and \(\mathcal B_\ell(x)\), \(x\) supplies
the scalar coordinate functions, while \(\ell\) specifies the coframe
in the tensor factors. Direct substitution shows that this tensor
is odd under \eqref{eq:reflection} and the simultaneous antipodal map.
The matching equations \eqref{eq:K1-coefficient-system} explain
these coefficients. The verification uses the derivative identities
\eqref{eq:OB-derivatives}--\eqref{eq:B-transverse}.

The compact formula and the program's coefficient table can also be
compared before evaluation on the manifold. Treat the eight quaternion
coordinates, \(r_0\), and the twelve coframe evaluations as independent
variables. Their quadratic evaluations on a formal tangent vector give
identical rational polynomials; the verifier compares every coefficient.
Polarization gives the tensor identity. The spatial-derivative test on a
specified curve is a separate check of the implementation.

The \(B\)-\(C\) identities
\eqref{eq:BC-j}--\eqref{eq:BC-k} are direct evaluations of
\eqref{eq:quadratic-connection-formula} and
\eqref{eq:torus-linearization}. Here we explain the global coverage
of the \(O\)-\(B\) check.

Write \(p=e^{\varphi i}\), \(q=e^{(\varphi+\theta)i}\), and define
\begin{align*}
 a&=-\frac{64}{455}\sin\theta\cos\theta,&
 b&=\frac{32}{6825}(41-30\sin^2\theta),\\
 d&=\frac{352}{6825}\sin\theta,&
 e&=\frac{1312}{6825}\cos\theta.
\end{align*}
Let \(A_{ab}\) denote the infinitesimal rotation taking \(e_a\) toward
\(e_b\). In the order \(A_{02},A_{03},A_{12},A_{13}\), the first
transverse derivatives of \(2\cQ(O,B)\) are
\begin{equation}\label{eq:OB-derivatives}
 \bigl(a\cos\varphi-b\sin\varphi,\
 b\cos\varphi+a\sin\varphi,\
 d\cos\varphi+e\sin\varphi,\
 -e\cos\varphi+d\sin\varphi\bigr).
\end{equation}
Restriction of \eqref{eq:K1} and application of
\eqref{eq:torus-linearization} give their negatives.

For a type-I torus, let \(u,w\) denote its constant left and right
great-circle axes. At \(T_0\), both equal \(i\).

The restrictions of the local expressions
\eqref{eq:A-local}--\eqref{eq:B-local} have transverse derivatives
\begin{align}
 \dot{\mathcal A}_\ell(x)(T,T)
 &=2\{x_0(\dot u_j-\dot w_j)
          +x_1(13\dot u_k/7-\dot w_k)\}
      -\tfrac{12}{7}\dot x_3,\label{eq:A-transverse}\\
 \dot{\mathcal B}_\ell(x)(T,T)
 &=2\{x_0(13\dot u_k/7+\dot w_k)
          -x_1(\dot u_j+\dot w_j)\}
      -\tfrac{12}{7}\dot x_2.\label{eq:B-transverse}
\end{align}

Here \(T\) is the great-circle tangent on the factor carrying \(\ell\),
normalized to have unit length in the round metric.
For \(A_{02},A_{03},A_{12},A_{13}\), respectively, the pairs
\((\dot u,\dot w)\) are
\[
 (k,-k),\quad(-j,j),\quad(j,j),\quad(k,k).
\]
They are independent of the common phase. Indeed, a fixed infinitesimal
\(SO(4)\)-action has the form \(\dot p=a p+p b\); differentiation
of \(u=p^{-1}\dot p_{\mathrm{circle}}\) and
\(w=\dot p_{\mathrm{circle}}p^{-1}\) gives constant commutator
variations of the axes.

The scalar coefficients in \(K_1^{\mathrm{alg}}\) restrict either to a coordinate on
the opposite factor or to \(r_0\) times a coordinate on the same
factor. In each case, the second derivative appearing in
\eqref{eq:torus-linearization} equals the negative of that coefficient.

Thus restriction and linearization contribute the factor
\(1/(2D_{\mathrm I})\) to the expressions
\eqref{eq:A-transverse}--\eqref{eq:B-transverse}. Substituting the four
axis pairs gives the negatives of \eqref{eq:OB-derivatives}.

When differentiating \(\cQ(O,B)\) under the diagonal action, the term
\(\cQ(O,\dot B)\) vanishes on \(\Tlift\). The reflection
\eqref{eq:reflection} keeps \(T_0\) totally geodesic for \(g_0+\tau O\),
so \(D_O(X,Y)\) is tangent to \(T_0\) for \(X,Y\in TT_0\).
As in Lemma~\ref{lem:pole-free}, \(\dd F_{1,O}=0\) on \(\Tlift\).
The tangential components of \(\dot B\) in the round circle coordinates
are constant, whence
\(\bigl(D_{\dot B}(X,Y)\bigr)^\top=0,\quad X,Y\in TT_0\).
The superscript \(\top\) denotes tangential projection.
The connection inner products and Schur term in
\eqref{eq:quadratic-connection-formula} therefore all vanish. This
shows that only the height functions need to be differentiated.

The verification treats \(\theta\) symbolically throughout and uses
the four exact common phases
\((\cos\varphi,\sin\varphi) \in\{(1,0),(0,1),(3/5,4/5),(4/5,3/5)\}\).
The axis variations are independent of the common phase. Moreover,
\(g_0\) and \(B\) are invariant under common left multiplication by
\(e^{i\varphi}\), while the height functions transform linearly in
\(\cos\varphi,\sin\varphi\). Thus the height-derivative expressions
are linear in these two functions. The four evaluations are also
injective on the larger space
\(\Span\{\cos\varphi,\sin\varphi,\cos3\varphi,\sin3\varphi\}\).
Indeed, for a polynomial with coefficients \(a_1,b_1,a_3,b_3\) in
this order, the first two zero evaluations give \(a_1=-a_3\),
\(b_1=b_3\). The remaining equations have coefficient matrix
\[
 \frac1{125}
 \begin{pmatrix}-192&144\\-144&192\end{pmatrix},
\]
which is nonsingular.
 Hence all four coefficients vanish.
The four rotations span the required transverse directions, and
reflection oddness gives zero value on \(\Tlift\). Thus
\(K_1^{\mathrm{alg}}\) satisfies condition~\eqref{eq:admissible-mixed} and the stated parities.

\subsection{Normal complements at the exceptional strata}
\label{app:normal-complements}

At the type-II representative \(p=1,q=j,u=i\), use the base curve
\[
 r(z)=\left(\frac{2z}{1+z^2},\,0,\,
                 \frac{1-z^2}{1+z^2},\,0\right)
\]
together with the eight graph variables of
\eqref{eq:quadratic-connection-formula}. The resulting nine-variable
Hessian \(\mathsf H_{\mathrm{II}}\) has kernel
\(\ker\mathsf H_{\mathrm{II}} =\Span\{\partial_{a_{E_k}}+\partial_{b_{F_k}}\}\).
This is the common-axis variation of \(u\) toward \(k\), with \(p,q\)
fixed. The rank calculation in Proposition~\ref{prop:clean-zero} shows that
there are no other kernel directions in this nine-dimensional slice.
Setting \(b_{F_k}=0\) is transverse to this kernel and selects an
eight-dimensional normal complement. Let \(A\) be the Hessian restricted
to that complement and \(r_{H_0}\) the corresponding differential of
the first unnormalized curvature coefficient.
The full nine
stationarity equations from \eqref{eq:schur-reduction} are checked
before restricting to the normal complement. The Schur term equals
\begin{equation}\label{eq:typeII-Schur}
 \tfrac12r_{H_0}^{T}A^{-1}r_{H_0}=\frac{5107}{151875}.
\end{equation}

The projection of \(\cZ_{\mathrm{II}}\) to \(M\) is the hypersurface
\(\langle p,q\rangle=0\). The \(z\)-direction supplies the one normal
direction missing from the fixed-base calculation.

At \(p=q=1\), use two base parameters
\[
 q(z_j,z_k)=
 \frac{1-z_j^2-z_k^2+2z_jj+2z_kk}{1+z_j^2+z_k^2}.
\]
Together with the eight graph variables these give ten variables. Their
Hessian \(\mathsf H_\Delta\) has kernel
\(\ker\mathsf H_\Delta =\Span\{\partial_{a_{E_j}}+\partial_{b_{F_j}},\, \partial_{a_{E_k}}+\partial_{b_{F_k}}\}\).
These are the two common-axis variations with the base point fixed.
Here the Hessian is restricted to the ten-dimensional slice.
The conditions \(b_{F_j}=b_{F_k}=0\) intersect
the kernel trivially and therefore give an eight-dimensional normal
complement.
In the order
\((a_{E_j},a_{E_k},a_{F_j},a_{F_k},
 b_{E_j},b_{E_k},b_{F_j},b_{F_k},z_j,z_k)\), the first-order minimizing displacement
for \(B+C\) is
\[
 \frac1{65}(-192,0,-64,0,64,0,0,0,0,-16).
\]
These coefficients solve the linear stationarity condition in
\eqref{eq:schur-reduction} on the specified complement and satisfy
all ten equations before restriction. Substitution of the metric
with correction \(K_2+K_3\) gives quadratic coefficient zero
and cubic coefficient \(115712/63375\).
The mixed quadratic term between the two base parameters is zero by
the circle isotropy rotating the \(j,k\)-plane of the background.

The same coefficient is obtained along the smooth curve
\[
 q(\tau)=\frac{1-a^2\tau^2}{1+a^2\tau^2}
            -\frac{2a\tau}{1+a^2\tau^2}k,\qquad a=\frac{16}{65},
\]
\[
 X_\tau=E_i+\frac{\tau}{65}(-96E_j-64F_j),\text{ and }
 Y_\tau=F_i+\frac{\tau}{65}(64E_j+96F_j).
\]
Composing the full spatial metric two-jets with this base curve
before applying the curvature recurrence gives the exact coefficients
\((0,0,0,115712/63375)\). The graph-stationarity derivatives vanish
as well. This is the smooth representative from
Lemma~\ref{lem:pole-free}, which extends across the diagonal and
antidiagonal.

\subsection{The generic cubic and the accompanying computations}\label{app:cubic-calculation}

For the generic cubic calculation, substitute
\eqref{eq:singular-displacement} into
\eqref{eq:cubic-reduction}, with \(h=B+C\), \(k=K_2+K_3\).
The zeroth, first, and second coefficients vanish exactly, and the
third is \eqref{eq:cubic-rational}. A separate use of
\eqref{eq:contracted-curvature} gives the same rational function
and checks the eight first-displacement stationarity equations.
Since the lower three coefficients vanish, dividing by the varying
Gram determinant divides the cubic coefficient by \(D_{\mathrm I}\).
The companion repository contains the two-jet definitions, curvature checks,
and a \texttt{README.md} mapping them to the paper.
After installing the pinned requirements, run
\texttt{python3 verification/verify.py} from the repository root.
Success means exit status zero and a final line beginning \texttt{VERIFIED:}.

Proposition~\ref{prop:mixed-existence} proves the properties of the
integral correction \eqref{eq:geometric-K1}; the code implements the
algebraic formula \eqref{eq:K1}. For either choice, the auxiliary bounds
and Poisson right-hand side are formed from the chosen tensor.

The exact programs verify the finite identities recorded in this appendix.
The contracted calculation uses the primary metric and tensor two-jets.
Separately differentiated coordinate calculations provide floating-point
consistency checks. The passage from these identities to positive curvature
on the whole Grassmann bundle is proved in
Sections~\ref{sec:background}--\ref{sec:positivity}.

\subsection{Coefficient selection and the remaining freedom}
\label{app:coefficient-choice}

We derive the coefficient relations used in Section~\ref{sec:tensors}
and give a numerical reconstruction of the background and height-parameter
choices. The amplitude estimates appear in
Sections~\ref{sec:amplitude-choice}--\ref{sec:complete-metric}.

\paragraph{The background parameter.}
For the family \eqref{eq:background-family}, we test
\(t\in\{1/10,1/5,3/10,2/5\}\) in increasing order and stop when
the reduced coefficient of \(H_0\) on type II becomes positive.
The full normal calculation at \(p=1,q=j,u=i\), with each corresponding
background, gives
\[
\begin{array}{c|ccc}
 t&1/10&1/5&3/10\\ \hline
 \cQ(H_0)|_{\cZ_{\mathrm{II}}}
 &-15488/16875&-7139/1080000&282877/2278125.
\end{array}
\]
This selects \(t=3/10\); the remaining conditions are those listed in
Section~\ref{sec:background-choice}.

\paragraph{The coefficient \(787/5880\).}
Let \(h\) be the global height tensor in the ansatz of
Section~\ref{sec:height-choice}, with the off-diagonal coefficient
\(a_0\) left free. Its tangential restriction to \(T_0\) is
\[
 h|_{TT_0}=\begin{pmatrix}
 \cos s&a_0(\cos s+\sin t)\\
 a_0(\cos s+\sin t)&\sin t
 \end{pmatrix}.
\]
The reflection fixing \(T_0\) preserves this metric path, so the torus
remains totally geodesic. Its first curvature variation vanishes, and
Codazzi shows that the differential of this variation in the plane
directions also vanishes. Together
with \(\cL h=0\) on the zero family, this gives zero first normal
displacement, initially on the regular locus and then by continuity.
Consequently its intrinsic quadratic coefficient is the reduced one.

In the coordinate frame \((\partial_s,\partial_t)\), the lower
connection variations on the torus are
\[
 \Gamma^{(1)}_{ss}=-\sin s\,(1/2,a_0)^T,\qquad
 \Gamma^{(1)}_{st}=0,\qquad
 \Gamma^{(1)}_{tt}=\cos t\,(a_0,1/2)^T.
\]
Since \(G_{\mathrm I}^{-1}=\frac1{30}
 \left(\begin{smallmatrix}29&19\\19&29\end{smallmatrix}\right)\),
the connection-variation formula \eqref{eq:connection-variation} gives
\begin{equation}\label{eq:height-free-coefficient}
 \cQ(h)|_{\Tlift}
 =-\frac{(\Gamma^{(1)}_{ss})^TG_{\mathrm I}^{-1}
                      \Gamma^{(1)}_{tt}}{D_{\mathrm I}}
 =\frac{76a_0^2+116a_0+19}{225}\sin s\cos t.
\end{equation}
For the two-coefficient correction in Section~\ref{sec:height-choice},
\eqref{eq:torus-linearization} gives
\(\cL K_0(\zeta_0,\zeta_1)|_{\Tlift} =\frac8{15}(\zeta_1-\zeta_0)\sin s\cos t\).
Thus the cancellation equation is
\begin{equation}\label{eq:K0-free-equation}
 \zeta_1-\zeta_0=-\frac{76a_0^2+116a_0+19}{120}.
\end{equation}
Its specialization at \(a_0=-2/7\) is \eqref{eq:K0-matching}.

To quantify the remaining freedom, keep the background and height
tensor fixed. Replacing
\((\zeta_0,\zeta_1)\) by
\((\zeta_0+\delta,\zeta_1+\delta)\) preserves
\eqref{eq:K0-free-equation}. For orthonormal height vectors \(e,f\)
and a supporting plane \(L\), direct use of
\eqref{eq:torus-linearization} shows that this changes the reduced
coefficient by
\[
 \frac{\delta}{D_{\mathrm I}}\cos\theta
       \langle e_{\parallel},f_{\parallel}\rangle
 =-\frac{8\delta}{15}\cos\theta
       \langle e_\perp,f_\perp\rangle,
\]
where the parallel projections are onto \(L\).
Its absolute value is at most \(4|\delta|\rho/15\).
By \eqref{eq:odd-lower-bound}, the new lower bound is at least
\((1/250-4|\delta|/15)\rho\). Requiring this to be at least
\(\rho/500\) gives \(|\delta|\leq3/400\). The chosen construction
corresponds to \(\delta=0\).

\paragraph{A reproducible numerical selection.}
Let \(O_{a_0}\) denote the height ansatz in Section~\ref{sec:height-choice}.
For each pair \((a_0,\zeta_0)\), determine \(\zeta_1\) by
\eqref{eq:K0-free-equation}. After diagonal \(SO(4)\)-normalization,
write the corrected coefficient for orthonormal height vectors \(e,f\) as
\[
 s_{a_0}(\theta)\rho
 +c_{a_0,\zeta_0}(\theta)\langle e_\perp,f_\perp\rangle
 +j_{a_0}(\theta)\langle Je_\perp,f_\perp\rangle.
\]
The scalar functions \(s,c,j\) are obtained from the frame calculus and
normal minimization; they are rational in \(z=\tan(\theta/2)\).
Since \(e_\perp\) and \(Je_\perp\) are orthogonal with equal length,
a lower bound for the coefficient of \(\rho\) is
\[
 s_{a_0}(\theta)-\tfrac12
       \sqrt{c_{a_0,\zeta_0}(\theta)^2+j_{a_0}(\theta)^2}.
\]
Their symmetries reduce the angular range to \([0,\pi/2]\).
Evaluate the minimum of this bound at \(513\) equally spaced angular
nodes, for an \(81\)-by-\(81\) uniform grid in \((a_0,\zeta_0)\in[-1,1]^2\).
Refine the best pair in the eight coordinate and diagonal directions,
starting with step \(0.025\). Retain improvements exceeding \(10^{-12}\),
halve the step when none occurs, and stop at step at most \(10^{-6}\),
keeping all candidates in the same box. This gives approximately
\[
 (a_0,\zeta_0)=(-0.2805573,-0.1875),\qquad
 \text{sampled margin}=0.00577854.
\]
For denominator bounds \(2,3,\ldots,32\), replace each coordinate by
its closest rational approximation with that denominator bound. The first
pair retaining at least three quarters of the sampled margin occurs at
bound seven and is \((-2/7,-1/5)\), with sampled margin approximately
\(0.00547401\). Equation~\eqref{eq:K0-free-equation} then gives
\(\zeta_1=-787/5880\).

For the resulting rational pair, the identity \eqref{eq:odd-identity}
and the estimate \eqref{eq:height-numerical-margin} give
\[
 \frac{8004517}{1778112000}>\frac1{250}
\]
uniformly for all angles, including the exceptional ones by smoothness.
The command \texttt{python3 verification/parameters.py} in the companion
repository reproduces this selection and verifies the rational identities.

\paragraph{The relative coefficient \(3/13\) in \(B\).}
Let \(B_{b_B}\) be \eqref{eq:B-tensor} with \(3/13\) replaced by \(b_B\).
It still satisfies \(\cL B_{b_B}=0\). At the two regular type-I
representatives \(p=1\), \(q=\cos\theta+i\sin\theta\), let
\(Q_1(b_B)\) and \(Q_2(b_B)\) denote its reduced quadratic coefficient
at \((\cos\theta,\sin\theta)=(0,1)\) and \((3/5,4/5)\), respectively.
Substitution in \eqref{eq:quadratic-connection-formula} gives
\begin{equation}\label{eq:B-selection-values}
\begin{aligned}
 Q_1(b_B)&=-\frac{1531b_B^2-2282b_B+1051}{4000},\\
 Q_2(b_B)&=-\frac{73017367b_B^2-37874994b_B+6457527}{37856250}.
\end{aligned}
\end{equation}
At both representatives, the eight plane directions form a normal
complement and the background Gram determinant is \(15/8\).
The formulas follow from the frame recurrences
\eqref{eq:koszul-engine}--\eqref{eq:curvature-engine} and the Schur
subtraction in \eqref{eq:quadratic-connection-formula}.
The dependence on \(b_B\) is quadratic.
Subtracting the polynomials in \eqref{eq:B-selection-values} gives
\begin{equation}\label{eq:B-selection-equation}
 Q_2(b_B)-\frac7{25}Q_1(b_B)
 =-\frac{6528787}{605700000}(13b_B-3)^2.
\end{equation}
Proportionality to \(\cos2\theta\) forces the left-hand side to vanish,
hence \(b_B=3/13\). Equation~\eqref{eq:BB-cancel} gives the resulting
identity for all angles.

\paragraph{Coefficients in the optional algebraic correction.}
The coefficients in \eqref{eq:A-local}--\eqref{eq:K1} can be found by
matching the four transverse derivatives in \eqref{eq:OB-derivatives}.
Replace \(13/7\) in \(W_\ell^\pm\) by
\(h\), and replace the coefficient \(-4/7\) of
\(x_3T_\ell\), respectively \(x_2T_\ell\), by \(d\).
Denote the resulting local tensors by \(A_{\ell;h,d}\) and
\(B_{\ell;h,d}\), and use the ansatz
\[
 \omega_1\{A_{\beta;h,d}(p)+B_{\alpha;h,d}(q)\}
 +\omega_2 r_0\{B_{\alpha;h,d}(p)+A_{\beta;h,d}(q)\}.
\]
Its parities hold for all coefficients. In the transverse derivatives
\eqref{eq:A-transverse}--\eqref{eq:B-transverse},
replace \(13/7\) by \(h\) and \(-12/7\) by \(3d\); the latter
factor occurs because \(T_\ell(T,T)=3\) on \(T_0\).
Matching the four required derivatives in \eqref{eq:OB-derivatives}
amounts to
\begin{equation}\label{eq:K1-coefficient-system}
\begin{aligned}
 \omega_1(1+h)&=\frac{44}{455},&
 \omega_1\left(1-\frac34d\right)&=\frac{22}{455},\\
 \omega_2(1+h)&=\frac{24}{91},&
 \omega_2\left(1-\frac34d\right)&=\frac{12}{91}.
\end{aligned}
\end{equation}
One obtains these equations by inserting the axis pairs
\((k,-k),(-j,j),(j,j),(k,k)\) into
\eqref{eq:A-transverse}--\eqref{eq:B-transverse} and equating the coefficients of
\(\cos\varphi,\sin\varphi\); the circle-coordinate second derivative
contributes the factor \(1/(2D_{\mathrm I})\) from \eqref{eq:torus-linearization}.
They are linear in the six products \(\omega_1\), \(\omega_1h\), \(\omega_1d\), \(\omega_2\), \(\omega_2h\), and \(\omega_2d\).
Solving \eqref{eq:K1-coefficient-system} leaves the one-parameter freedom
\[
 h=1-\frac32d,\qquad
 \omega_1=\frac{22}{455(1-3d/4)},\qquad
 \omega_2=\frac{12}{91(1-3d/4)},\qquad d\ne\frac43.
\]
For \(d=-4/7\), we obtain \(h=13/7\), \(\omega_1=11/325\), and
\(\omega_2=6/65\), as in \eqref{eq:K1}.